\documentclass[11pt]{article}
\usepackage{pgfplots}
\pgfplotsset{compat=1.15}
\usepackage{mathrsfs}
\usetikzlibrary{arrows}
\usepackage{amscd,amsmath,amsfonts,amssymb,amsthm,cite,mathrsfs,color} 
\usepackage{dsfont} \usepackage{leftidx}
\usepackage{tikz} \usepackage{enumerate}
\usepackage{appendix}

\usepackage{hyperref}

\usepackage[a4paper,left=25mm, right=25mm, top=25mm, bottom=25mm]{geometry}
\usepackage{graphicx} \usepackage{xcolor}

\numberwithin{equation}{section} 

\usepackage{amsthm}

\theoremstyle{plain}
\newtheorem{theorem}{Theorem}[section]
\newtheorem{lemma}[theorem]{Lemma}

\theoremstyle{definition}

\newtheorem{remark}{Remark}[section]

\newtheorem*{acknow}{Acknowledgments}
\begin{document}
	
	\title{\bf{Local Statistics of Singular Values for Products of Truncated Unitary Matrices  }}
	
	\author{
		Yandong Gu\footnotemark[1],   ~Dang-Zheng Liu\footnotemark[1] ~ 
	}
	\renewcommand{\thefootnote}{\fnsymbol{footnote}}
	\footnotetext[1]{School of Mathematical Sciences, University of Science and Technology of China, Hefei 230026, P.R.~China. E-mail: gd27@mail.ustc.edu.cn, dzliu@ustc.edu.cn
		
	}

	\maketitle
\begin{abstract}
This paper investigates  local spectral statistics of singular values for many products of independent large rectangular matrices, sampled from the ensemble of truncated unitary matrices with the invariant Haar  measure. Our main contribution establishes a universal three-phase transition in these statistics, demonstrating an    interpolation between  GUE statistics and classical Gaussian behavior.  While such transition was  previously known for products of complex Gaussian matrices\cite{ABK19}\cite{LWW23},  the current work provides the complete characterization in the truncated unitary matrix setting.

 
 
  

\end{abstract}

\section{Introduction and main results}

\subsection{Introduction}
The spectral properties for products of $M$ independent random matrices of size $n\times n$ have attracted significant interest due to their applications in wireless communications \cite{TV04}, statistical physics relating to chaotic dynamical systems\cite{Cpv93}, and free probability theory\cite{MS17}.   
The study on products of random matrices  dates back to    the seminal articles
by Bellman \cite{Bel54} in 1954 and further by Furstenberg and Kesten \cite{FK60} in 1960, in which
classical limit theorems in probability theory were obtained for singular values for the product under certain assumptions on the matrix entries,  
when $M$ goes to infinity  but   $n$ is any fixed integer.  
However, the more recent interest in products of random matrices lies in  the spectral  statistics   of singular values as the matrix size $n$ goes to infinity,
like a single random matrix \cite{AI15}\cite{AIK13}\cite{AKW13}\cite{AGT10}\cite{FL16}\cite{kks16}\cite{ks14}\cite{LWZ16},  when   the product number $M$ is any     fixed integer.
Luckily,   in the special case for   products of  i.i.d. complex Ginibre matrices, Akemann, Kieburg and Wei  \cite{AKW13}  gave an  explicit  form of  singular value density  and  proved  that they  indeed  form   determinantal point processes.   Furthermore,  Kuijlaars  and Zhang  \cite{KZ14}
established a double  integral representation for  the  correlation kernel.    With the help of this determinantal  structure, for any fixed  $M$ and as $n \to \infty$,  the 
second  named author    with  Wang and Zhang proved the sine and Airy kernels%
 \cite{LWZ16}.  In an opposite direction, for any fixed $n$ and as $M \to \infty$, Akemann, Burda and
Kieburg proved in\cite{ABK14} that $n$ finite-size Lyapunov exponents for the product  are asymptotically
independent Gaussian random variables.

A fundamental question naturally emerges regarding the spectral  behaviors when both  the depth (time size) parameter 
$M$ and  the width   (space size) parameter 
$n$ tend to infinity--a regime central to understanding universal patterns in random matrix theory. In  the complex Ginibre setting,  the second named author,  Wang and  Wang  \cite{LWW23}  rigorously established that the local singular value statistics undergo a phase transition as the   depth-to-width ratio (DWR), defined as the number of  layers (products)
 relative to the layer width (matrix size)
$M/n$
 varies from 0 to $\infty$.  This   crossover phenomenon was also independently discovered by Akemann, Burda, and Kieburg \cite{ABK19} in the physical literature.
 Specifically,  as $M+n\to \infty$  there exists a universal three-phase transition for the correlation kernel.
  \begin{itemize}
 \item  {\bf{High DWR (Deep \& Narrow: $M/n \to \infty$)}}: Normality,   Gaussian fluctuation appears;

\item {\bf{Moderate DWR (Balanced: $M/n  \to \gamma \in (0,\infty)$) }}: Criticality,   critical kernels  ocurrs;  

\item  {\bf{Low DWR (Shallow \& Wide : $M/n \to 0$)}}: GUE statistics,   Airy and sine kernels appear at the soft edge and in the bulk, respectively.
\end{itemize}

A related double-limit phenomenon has been extensively investigated across various random matrix product ensembles, revealing universal transition patterns in mathematical aspects of spectral statistics. Notably: 
 \begin{itemize}
 \item[(i)] For products of i.i.d. non-Hermitian matrices, Liu and Wang \cite{LW24} established a phase transition in local complex eigenvalue statistics between Ginibre  and Gaussian  regimes.

\item[(ii)]  For products of complex Ginibre or truncated unitary matrices, Jiang and Qi \cite{Jq17} established analogous phase transition  for  absolute values of  complex eigenvalues (spectral radius),  from Gumbel to Gaussian.

\item[(iii)]   Gorin and Sun \cite{GS22} investigated global fluctuations of singular values for products of invariant random matrices, proving that the corresponding height functions converge to an explicitly Gaussian field.

\item[(iv)]
Ahn \cite{AA22} established for the $\beta$-Jacobi product process: (1) a limit shape theorem, (2) Gaussian global fluctuations with explicit covariances (generalizing $\beta=1,2,4$ cases). In Moderate DWR regime, $\beta=2$ local fluctuation results imply right-edge convergence to a process interpolating between Airy point process  and deterministic configurations from \cite{ABK19}\cite{LWW23}. Separately, \cite{AP23} showed that Lyapunov exponents for mixed products of Ginibre and truncated Haar matrices yield identical 'picket fence' statistics.

\item[(v)]  For products of real Ginibre matrices \cite{HP21} and matrices with i.i.d. entries \cite{HJ25}, non-asymptotic analyses of singular values and Lyapunov exponents are established by Hanin, Paouris and Jiang. These include an quantitative convergence rate for the empirical measure of squared singular values to the uniform distribution on $[0,1]$.

\item[(vi)] In directed last-passage percolation, Berezin and Strahov \cite{BE25} established a surprising connection to truncated orthogonal/unitary ensembles. Remarkably, the critical kernel \eqref{kernelcriedge} emerges at the soft edge and exhibits a hard-to-soft transition.

\end{itemize}
These universal patterns suggest interesting phase transitions and also   a profound connection between random matrix products  and  other probabilistic models.

 In this paper, we focus on  the products of  truncated unitary matrices and study local statistics of singular values.  Our  product model  is   defined as follows.
 
  \begin{itemize}
 \item   {\bf{Truncated  product model:}}
  $Y_M=T_M\cdots T_1$,  where
 each \( T_j \) is a left-upper \( (n+v_{j}) \times(n+v_{j-1})  \) truncation of Haar-distributed unitary matrices \( U_j \) of size \( m_j \geq n+v_{j-1} \),  with $v_0=0$ and $ v_1,\cdots,v_M$ are non-negative integers. 
 
  \end{itemize}


 Kieburg, Kuijlaars, and Stivigny \cite{kks16} established that for products of rectangular truncated unitary matrices, the squared singular values follow a determinantal point process. The associated correlation kernel admits an explicit representation as a double contour integral.
Moreover,  \cite{CKW15} characterized when the eigenvalues of the Hermitian matrices $Y^*_MY_M$ form a determinantal point process, if and only if the dimensional constraint
\begin{equation}\label{dcon}
n \leq \sum_{j=1}^{M}(m_j - n - v_j)
\end{equation}
is satisfied. In such cases, the correlation kernel of the log-transformed matrices $\log(Y^*_MY_M)$ takes the form
\begin{equation}\label{logkn}
K_n(x,y) = \int_{\mathcal{C}}\frac{ds}{2\pi i}\oint_{\Sigma}\frac{dt}{2\pi i}\frac{e^{xt-ys}}{s-t}\prod_{j=0}^M\frac{\Gamma(s+n+v_j)\Gamma(t+m_j)}{\Gamma(t+n+v_j)\Gamma(s+m_j)},
\end{equation}
with $m_0 = 0$ by convention. Here
 $\mathcal{C}$ is a positively oriented Hankel contour in the left half-plane, encircling $(-\infty,-1]$ while starting and ending at $-\infty$, and 
    $\Sigma$ is a closed contour enclosing $[0,n-1]$ chosen disjoint from $\mathcal{C}$.
    
 Under  a  special assumption that   all ratios $m_j/n \to 2$,   the free multiplication convolution from free probability theory guarantees that  as $n \to \infty$ the empirical spectral distribution of $Y^*_{M}Y_{M}$ converges weakly to a limiting measure $\mu_M$. The density of $\mu_M$ and its edge behavior at the endpoints of the support interval
$
 \big[0, (M+1)^{M+1}2^{-(M+1)}M^{-M}\big]
 $ are   explicitly  characterized in \cite{NT20}. 
 This result parallels the well-established parameterization via free probability \cite{HM13} for the spectral distribution of squared singular values of products of Ginibre matrices, with complementary insights from characteristic polynomial approaches developed in \cite{Ne14}. These global spectral descriptions play a central role in etablishing  bulk statistics for the product models\cite{LWW23}\cite{LWZ16}.

 \vspace{0.5em}
 \noindent \textbf{Notations:} Throughout this work,
 \begin{itemize}
 	\item $\psi(z) = \Gamma'(z)/\Gamma(z)$ denotes the digamma function;
 	\item $[x]$ represents the floor function (greatest integer less than or equal to $x$);
 	\item $\Re z$ and $\Im z$ denote the real and imaginary parts of $z$, respectively.
 \end{itemize}

\subsection{Main results}

The main goal of this paper   is to  investigate the local statistics of singular values for the   truncated  product model and to  identify   a phase transition in these local statistics
as the   depth-to-width ratio (DWR), defined as the modified  ratio of the number of   products 
 relative to the  matrix size,
\begin{equation}
    \Delta_{M,n} = \sum_{j=0}^M \frac{1}{n+v_j} - \sum_{j=1}^M \frac{1}{m_j}
\end{equation}
 varies from 0 to $\infty$. 

%



\begin{theorem} {\bf({Normality})}\label{supthm}
	Suppose that $\lim\limits_{M+n \to \infty} \Delta_{M,n}=\infty$. Given any $k \in \{1,2,\cdots,n\}$, let \begin{equation}\rho_{M,n}(k)=\Big(\sum_{j=0}^M\psi'(n+v_j+1-k)-\sum_{j=1}^M\psi'(m_j+1-k)\Big)^{\frac{1}{2}},
	\end{equation} and
	\begin{equation}
		g(k;\xi)=\sum_{j=0}^M\psi(n+v_j+1-k)-\sum_{j=1}^M \psi(m_j+1-k)+\xi\rho_{M,n}(k),
	\end{equation}
then uniformly for $\xi,\eta$ in a compact subset of $\mathbb{R}$, we have 
	\begin{equation}
		\lim_{M+n \to \infty}\rho_{M,n}(k) e^{(k-1)(\xi-\eta)\rho_{M,n}(k)}K_n\big(g(k;\xi),g(k;\eta)\big)=\frac{1}{\sqrt{2\pi}}e^{-\frac{\eta^2}{2}}.
	\end{equation}
\end{theorem}
In order to state the critical results, we need to intoduce two kernels introduced in \cite{ABK19} \cite{LWW23}, which correspond to the bulk and edge limits respectively. For $\gamma \in (0,\infty)$, we define 
\begin{equation}\label{kercribulk}
	K_{\mathtt{crit}}^{\mathtt{(bulk)}}(x,y;\gamma)=\frac{1}{\sqrt{8\pi \gamma}}\int_{-1}^{1}e^{\frac{(\pi w-iy)^2}{2\gamma}}\theta\big(\frac{\pi w-ix}{2\pi},\frac{i\gamma}{2\pi}\big)dw,
\end{equation}
 and
\begin{equation}\label{kernelcriedge}
	K_{\mathtt{crit}}^{\mathtt{(edge)}}(x,y;\gamma)=\int_{1-i\infty}^{1+i\infty}\frac{ds}{2\pi i}\oint_{\Sigma_{-\infty}}\frac{dt}{2\pi i}
	\frac{1}{s-t}\frac{\Gamma(t)}{\Gamma(s)}\frac{e^{\frac{\gamma s^2}{2}-ys}}{e^{\frac{\gamma t^2}{2}-xt}},
\end{equation}
where the anticlockwise contour $\Sigma_{-\infty} \subset \{z\in \mathbb{C} \mid \Re z <1\}$, starts from $-\infty-i\epsilon$, encircles $\{0,-1,-2,\dots\}$, and then goes to $-\infty+i\epsilon$ for some $\epsilon>0$. 
The Jacobi theta function  is given by 
\begin{equation}
	\theta(z,\tau)=\sum_{n=-\infty}^{\infty}e^{\pi i n^2\tau+2\pi inz},\quad \Im \tau >0, z\in \mathbb{C}.
\end{equation}

\begin{theorem}{\bf(Criticality)}\label{critthm}
Suppose that $\lim\limits_{M+n \to \infty} \Delta_{M,n}=\gamma \in (0,\infty)$.\\
$\mathtt{(i)}$ {\bf(Critical bulk limit)} For any given $u\in(0,1)$, let \begin{equation}\label{gamma'}
\gamma'=\lim\limits_{M+n \to \infty}\sum_{j=0}^M\frac{1}{n+v_j-nu}-\sum_{j=1}^M\frac{1}{m_j-nu},
\end{equation} and 
\begin{equation}\label{cribgf}
	\begin{split}g(\xi)=&\sum_{j=1}^M\log\Big(\frac{n+v_j-nu}{m_j-nu}\Big)+\log\big(\frac{1-u}{u}\big)\\&+\Big(\sum_{j=0}^M\frac{1}{n-nu+v_j}-\sum_{j=1}^M\frac{1}{m_j-nu}\Big)\big(nu-[nu]-\frac{1}{2}\big)+\xi,\end{split}
\end{equation}
then uniformly for $\xi,\eta$ in a compact subset of $\mathbb{R}$, we have 
\begin{equation}
	\lim_{M+n \to \infty}e^{(g(\xi)-g(\eta)[nu]}K_n\big(g(\xi),g(\eta)\big)=	K_{\mathtt{crit}}^{\mathtt{(bulk)}}(\xi,\eta;\gamma').
\end{equation}
$\mathtt{(ii)}$ {\bf(Critical  edge limit)}   For $ u=0$, 
let \begin{equation}
	g(\xi)=\sum_{j=0}^M\log(n+v_j)-\sum_{j=1}^M\log(m_j)-\frac{1}{2}\big(\sum_{j=0}^M\frac{1}{n+v_j}-\sum_{j=1}^M\frac{1}{m_j}\big)+\xi,
\end{equation}
then uniformly for $\xi,\eta$ in a compact subset of $\mathbb{R}$, we have 
\begin{equation}
	\lim_{ n\to \infty}K_n\big(g(\xi),g(\eta)\big)=	K_{\mathtt{crit}}^{\mathtt{(edge)}}(\xi,\eta;\gamma).
	\end{equation}

\end{theorem}
For the third regime $\lim\limits_{M+n \to \infty}\ \Delta_{M,n}=0$, we need to introduce the scaling parameter 
\begin{equation}
\rho_{M,n}=\Big(\frac{1}{2}\sum_{j=0}^M\frac{(m_j+z_0)^2-(n+v_j+z_0)^2}{(m_j+z_0)^2(n+v_j+z_0)^2}\Big)^{-\frac{1}{3}},
\end{equation}
and a spectral parameter
\begin{equation}
	\lambda_M=\prod_{j=0}^M\frac{n+v_j+z_0}{m_j+z_0},
\end{equation}
where $z_0$ is the unique positive solution of the rational equation
\begin{equation}\label{defz_0}
	\sum_{j=0}^M\Big(\frac{1}{n+v_j+z}-\ \frac{1}{m_j+z}\Big)=0.
\end{equation}
We also give a definition of the celebrated Airy kernel by
\begin{equation}
	\begin{split}
		K_{\mathtt{Ai}}(x,y)
		&=\frac{1}{(2 \pi i)^2}\int_{\gamma_R}du\int_{\gamma_L}d\lambda \frac{e^{u^3/3-xu}}{e^{{\lambda}^3/3-y\lambda}}\frac{1}{u-\lambda},
	\end{split}
\end{equation}
where $\gamma_R$ and $\gamma_L$ are symmetric with respect to the imaginary axis, and $\gamma_R$ is a contour in the right-half plane going from $e^{-\frac{\pi i}{3}}\cdot \infty$ to $e^{\frac{\pi i}{3}}\cdot \infty$, see e.g.\cite{AGZ10}.
\begin{theorem}{\bf(GUE  statistics)}	Suppose that $\lim\limits_{M+n \to \infty} \Delta_{M,n}=0$.\\
$\mathtt{(i)}$	{\bf(GUE edge statistics)}\label{subthm}
 Let 
	\begin{equation}\label{subcgd}
		g(\xi)=\log \lambda_M+\frac{\xi}{\rho_{M,n}},
	\end{equation}
then uniformly for $\xi,\eta$ in a compact subset of $\mathbb{R}$, we have 
\begin{equation}
	\lim_{M+n \to \infty}
\frac{1}{\rho_{M,n}}e^{-(\xi-\eta)\frac{z_0}{\rho_{M,n}}}	K_n\big(g(\xi),g(\eta)\big)=K_{\mathtt{Ai}}(\xi,\eta).
\end{equation}
$\mathtt{(ii)}${\bf(GUE bulk statistics)}\label{TBU}For $0<a<M$, we assume that 
\begin{equation}\label{assumptionss}
	\frac{m_j}{n} \to 1+\frac{1}{a},\quad \frac{v_j}{n}\to 0,\quad \text{for $j=1,2,\dots,M$}.
\end{equation} 
For $\theta \in \big(0,\frac{\pi}{M+1}\big)$ , let 
\begin{equation}
	g(\xi)=\log\Big(\frac{a^M}{(a+1)^{M+1}}\Big)+V_M(\theta)+\frac{\xi}{\rho_{M,n}(\theta)},
\end{equation}
where\begin{equation}
V_M(\theta)=(M+1)\log\sin((M+1)\theta)-M\log\sin (M\theta)-\log\sin (\theta),\end{equation}and 
\begin{equation}
	\rho_{M,n}(\theta)=\frac{n\sin( (M+1)\theta)\sin(M\theta)\sin(\theta)}{\pi\Big((a+1)^2\sin^2(M\theta)\sin^2(\theta)+(a\cos(M\theta)\sin(\theta)-\sin(M\theta)\cos(\theta))^2\Big)},
\end{equation}
then uniformly for $\xi$ and $\eta$ in any compact subset of $\mathbb{R}$, we have 
	\begin{equation}
		\lim_{n \to \infty}e^{-\pi(\xi-\eta)\big(\cot(\theta)-\frac{a\sin((M+1)\theta)}{(a+1)\sin(M\theta)\sin(\theta)}\big)}\frac{1}{n\rho(\theta)}K_n\big(g(\xi),g(\eta)\big)=\frac{\sin(\pi(\xi-\eta))}{\pi (\xi-\eta)}.
	\end{equation}
 
\end{theorem}

\begin{remark}
The three-phase transition of singular values  remains valid for the following models.
\begin{itemize}
	\item \textbf{Ginibre product model:}
	$Y_M = X_M \cdots X_1$, where each $X_j$ is a complex Ginibre matrix of size $(n + v_{j}) \times (n + v_{j-1})$ with $v_0 = 0$ and  $v_1, \dots, v_M \geq 0$.
\end{itemize}
In this model, the modified DWR is defined as \begin{equation}
\Delta_{M,n} = \sum\limits_{j=0}^{M} \frac{1}{n + v_j}.
\end{equation}
The critical result ($\Delta_{M,n}\to \gamma \in(0,\infty)$) at the soft edge was first established in \cite{LWW23}.
\begin{itemize}
	\item \textbf{Inverse Ginibre product model:}
$\Pi_{M,K}=X_M\cdots X_1(Y_K \cdots Y_1)^{-1}$ where $X_j,j=1,\dots,M$ and $Y_k,k=1,\dots,K$ are i.i.d. complex Ginibre
matrices with size $(n+v_j )\times (n+v_{j-1})$
and $(n + u_k) \times (n + u_{k-1})$, respectively. With $v_0=u_0=u_K=0$ and  $v_j,u_k \geq 0$, thus, $\Pi_{M,K}$ is a rectangular matrix of size $(n + v_M) \times n$. 
 \end{itemize}
  The squared singular values of $\Pi_{M,K}$ forms a determinantal process\cite{F0R14}.
In this model, the modified DWR should be  defined as \begin{equation}
	\Delta_{M,K,n} = \sum\limits_{j=0}^{M} \frac{1}{n + v_j}+\sum\limits_{k=0}^{K} \frac{1}{n + u_j}.
\end{equation}
\end{remark}
   
 The proofs of the aforementioned theorems  build on  an analogous procedure  to that of \cite{LWW23}, which addresses three distinct phase regimes. 
 Our  central contribution   is the derivation of a comprehensive parameter, the modified  DWR, which provides a unified characterization of all three observed phenomena.
 Furthermore, in the low DWR regime—particularly for the special case involving two distinct matrix sizes—we establish edge and bulk  limits through new  parameterizations of the spectrtal edge  and  the limiting  density. This greatly extend the results for the classical complex Wishart ensemble.
 Finally, our approach naturally extends to other random matrix product models. Notable examples include products of complex Ginibre rectangular matrices, as well as products of Ginibre and inverse Ginibre matrices, the latter of which encompasses the spherical ensemble as a special case. 
 
The rest of this paper is organized as follows. In  next Section \ref{sect.2} we prove Theorems   \ref{supthm}, \ref{critthm} and  \ref{subthm} $\mathtt{(i)}$. In Section \ref{sect3} we prove Theorem \ref{subthm} $\mathtt{(ii)}$ and  discuss a simplified model through  a parameterization expression of the limiting  density.

\section{Proofs of Theorems    \ref{supthm}, \ref{critthm} and  \ref{subthm} $\mathtt{(i)}$ }\label{sect.2}

In the proofs, we use $\psi(z)$ to denote the digama function, which admits a series representation for $z\neq 0,-1,-2,\cdots$
\begin{equation}
	\psi(z)=-\gamma_0+\sum_{n=0}^{\infty}\big(\frac
	{1}{n+1}-\frac{1}{n+z}\big),\quad \psi'(z)=\sum_{n=0}^{\infty}\frac{1}{(n+z)^2},
\end{equation}
where $\gamma_0$ is the Euler constant; see\cite{OLBC10}. By the Stirling's formula, we have as $z \to \infty$ in the sector $|\arg(z)|\leq \pi-\epsilon$ for all $\epsilon>0$, uniformly 
\begin{equation}\label{eslogg}
	\log\Gamma(z)=\big(z-\frac{1}{2}\big)\log(z)-z+\log\sqrt{2\pi}+\frac{1}{12z}+O\big(\frac{1}{z^2}\big),
\end{equation}
and 
\begin{equation}
	\psi(z)=\log(z)-\frac{1}{2z}+O\big(\frac{1}{z^2}\big).
\end{equation}

\subsection{Proof of Theorem \ref{supthm}}
By the assumption on $M$ and $n$ in this setting, we know that $M$ must tend to infinity and $n$ may tend to infinity or be a
finite number, whenever $M + n \to \infty$.

\begin{proof}[Proof of Theorem \ref{supthm} ]

Firstly, we recall $\rho_{M,n}(k)$ and $g(k;\cdot)$ in Theorem \ref{supthm}, simple calculations give us 
\begin{align}\label{supker}
	e^{(k-1)(\xi-\eta)\rho_{M,n}(k)}K_n(g(k;\xi),g(k;\eta))=\int_{\mathcal{C}}\frac{ds}{2 \pi i} \oint_{\Sigma}\frac{dt}{2\pi i} \frac{\Gamma(t)}{\Gamma(s)}\frac
	{1}{s-t}\frac{e^{F(t;k)}}{e^{F(s;k)}}\frac{e^{(t+k-1)\rho_{M,n}(k)\xi}}{e^{(s+k-1)\rho_{M,n}(k)\eta}},
\end{align}
where 
\begin{align}
	F(t;k)=&\Big( \sum_{j=0}^M\psi(n+v_j+1-k)-\sum_{j=1}^M \psi(m_j+1-k)\Big)t\notag \\
	&-\sum_{j=0}^M \log \frac{\Gamma(t+n+v_j)}{\Gamma(n+v_j)}+\sum_{j=1}^{M} \log\frac{\Gamma(t+m_j)}{\Gamma(m_j)}.
\end{align}

Next, we set out the contours for $s$ and $t$.
For $s$, define a vertical contour passing beside $1-k$ as 
\begin{align}
	&	\mathcal{L}_{1-k}=\Big\{1-k+\frac{2}{\rho_{M,n}(k)}+iy\mid y\in \mathbb{R}\Big\},\\
	&\mathcal{L}_{1-k}^{\text{local}}=\Big\{z \in \mathcal{L}_{1-k} \mid |\Im z|\leq n^{1/4}\rho_{M,n}^{-3/4}(k)\Big\},\quad  \mathcal{L}_{1-k}^{\text{global}}=\mathcal{L}_{1-k} \setminus\mathcal{L}_{1-k}^{\text{local}}.
\end{align}
For $t$,
we define the following positively oriented contours.
For any $a\in (-n+1,1)$,
\begin{equation} 
\begin{split}
\Sigma_{-}(a)=&\Big\{a-\frac{2-i}{4}t \mid t \in [0,1]\Big\} \cup
\Big\{a-\frac{2+i}{4}+\frac{2+i}{4}t \mid t \in [0,1]\Big\}
 \cup \\
 & \Big\{-t+\frac{i}{4} \mid t \in [\frac{1}{2}-a,n-\frac{1}{2}]\Big\} \cup \Big\{t-\frac{i}{4} \mid t \in [-n+\frac{1}{2},a-\frac{1}{2}]\Big\}\cup\\ &  \Big\{-n+\frac{1}{2}-it \mid t\in [-\frac{1}{4},\frac{1}{4}]\Big\}.
\end{split}	
\end{equation}
Similarly, for $b\in(-n+1,\frac{1}{2})$,
\begin{equation} 
	\begin{split}
		\Sigma_{-}(b)=&\Big\{b+\frac{2-i}{4}t \mid t \in [0,1]\Big\} \cup
		\Big\{b+\frac{2+i}{4}-\frac{2+i}{4}t \mid t \in [0,1]\Big\}
		\cup \\
		& \Big\{t-\frac{i}{4} \mid t \in [b+\frac{1}{2},1]\Big\} \cup \Big\{-t+\frac{i}{4} \mid t \in [-1,-b-\frac{1}{2}]\Big\}\cup\\ &  \Big\{1+it \mid t\in [-\frac{1}{4},\frac{1}{4}]\Big\}.
	\end{split}	
\end{equation}
 Let $\Sigma_0(1-k)$ be the positively oriented circle centered at $1-k$ with radius $\rho_{M,n}^{-1}(k)$,
 then we choose the contour $\Sigma$ as the union of $\Sigma_0(1-k)$ and $\Sigma_{-}(\frac{1}{2}-k) \cup \Sigma_{+}(\frac{3}{2}-k)$ with $k\in \{1,2,\dots,n\}$, while $\Sigma_{+}(\frac{3}{2}-k)$ is set to be empty if $k = 1$.\\
Accordingly, we divide the integral \eqref{supker} into two parts,
\begin{equation}
	e^{(k-1)(\xi-\eta)\rho_{M,n}(k)}K_n(g(k;\xi),g(k;\eta))=I_1+I_2,
\end{equation}
where \begin{align}
	&	I_1^{*}=\int_{\mathcal{L}^{*}_{1-k}}\frac{ds}{2\pi i}\oint_{\Sigma_0(1-k)}\frac{dt}{2\pi i}\frac{\Gamma(t)}{\Gamma(s)}\frac
	{1}{s-t}\frac{e^{F(t;k)}}{e^{F(s;k)}}\frac{e^{(t+k-1)\rho_{M,n}(k) \xi}}{e^{(s+k-1)\rho_{M,n}(k) \eta}},\end{align}
here $*=\text{local, global or blank}$, and 
\begin{align}	&I_2=\int_{\mathcal{L}_{1-k}}\frac{ds}{2\pi i}\oint_{\Sigma_{-}(\frac{1}{2}-k) \cup \Sigma_{+}(\frac{3}{2}-k)}\frac{dt}{2\pi i}\frac{\Gamma(t)}{\Gamma(s)}\frac
	{1}{s-t}\frac{e^{F(t;k)}}{e^{F(s;k)}}\frac{e^{(t+k-1)\rho_{M,n}(k) \xi}}{e^{(s+k-1)\rho_{M,n}(k) \eta}}.
\end{align}
After making  change of variables
\begin{equation}\label{cov}
	s=1-k+\frac{\sigma}{\rho_{M,n}(k)},\quad t=1-k+\frac{\tau}{\rho_{M,n}(k)},
\end{equation}
we obtain
\begin{equation}
	\frac{e^{(t+k-1)\rho_{M,n}(k) \xi}}{e^{(s+k-1)\rho_{M,n}(k) \eta}}=\frac{e^{\xi \tau}}{e^{\eta \sigma}}.
\end{equation}
At the same time, for $\sigma$ and $\tau$ in a compact subset of $\mathbb{C}\setminus 0$ and  $\Delta_{M,n} \to \infty$, we have 
\begin{equation}
	\frac{\Gamma(t)}{\Gamma(s)}=\frac{\sigma}{\tau}\big(1+C_{\tau,\sigma}(|\tau|+|\sigma|)\rho_{M,n}^{-1}(k)\big),\quad \frac{1}{s-t}=\frac{\rho_{M,n}(k)}{\sigma-\tau},
\end{equation}
where $C_{\tau,\sigma}$ is bounded.

So the remaining task is to obtain asymptotic estiamtes of $I_1=I_1^{\text{local}}+I_1^{\text{global}}$ and $I_2$ as $\Delta_{M,n} \to \infty$. The key point here is to analyze the propoties of the functions $F(t;k)$ and $F(s;k)$.
To estimate the $I_1^{\text{local}}$, 
we need some properties of $F(t;k)$
\begin{align}
	F'(t;k)=&\sum_{j=0}^M\big(\psi(n+v_j+1-k)-\psi(t+n+v_j)\big) \notag \\
	&	+\sum_{j=1}^M \big(\psi(m_j+t)-\psi(m_j+1-k)\big),
\end{align}
and 
\begin{equation}
	F''(t;k)=-\sum_{j=0}^M\psi'(t+n+v_j)+\sum_{j=1}^M\psi'(t+m_j),
\end{equation}
and 
\begin{equation}F'(1-k;k)=0,\quad F''(1-k)=-\rho_{M,n}^2(k).
\end{equation}
When $t \in \Sigma_0(1-k)$, by the Taylor expansion with respect to $t$ at $1-k$, we have uniformly for all $\tau=\tau'n^{1/4}\rho_{M,n}^{1/4}(k)$  under the change of variables \eqref{cov}, where $\tau'$ is in a compact subset of $\mathbb{C}$,
\begin{equation}\label{Ftestimate}
	\begin{split}
		F(t;k)=&F(1-k;k)+\frac{1}{2}F''(1-k;k)\big(\frac{\tau}{\rho_{M,n}(k)}\big)^2+O\Big(\frac{\tau^3}{n\rho_{M,n}(k)}
	\Big)\\
          &=F(1-k;k)-\frac{\tau^2}{2}+C(\tau')n^{-1/4}\rho_{M,n}^{-1/4}(k),
\end{split}
\end{equation}
where $C(\tau')$ is bounded in $\mathbb{C}$.
For $s \in \mathcal{L}^{\text{local}}_{1-k}$, the estimate \eqref{Ftestimate} also holds true for $F(s;k)$. Combining  \eqref{Ftestimate} and some other elementary estimates for $F(t;k)$ and $F(s;k)$, we have
\begin{equation}
	I_1^{\text{local}}=\frac{1}{\rho_{M,n}(k)}\int_{2-in^{1/4}\rho_{M,n}^{1/4}(k)}^{2+in^{1/4}\rho_{M,n}^{1/4}(k)}\frac{d\sigma}{2\pi i}\oint_{|\tau|=1}\frac{d\tau}{2\pi i}\frac{e^{-\frac{\tau^2}{2}+\xi\tau}}{e^{-\frac{\sigma^2}{2}+\eta\sigma}}\frac{\tau}{\sigma}\frac{1}{\sigma-\tau}\Big(1+O(n^{-1/4}\rho_{M,n}^{-1/4}(k))\Big).
\end{equation}

Next, we will prove that  $I_1^{\text{global}}$ and $I_2$ can be neglected in the asymptotic analysis.
For $s\in \mathcal{L}_{1-k}^{\text{global}}$, $s=(1-k)+(2+iy)\frac{1}{\rho_{M,n}(k)}$ for $|y|>n^{1/4}\rho_{M,n}^{1/4}(k)$, 
we use 
\begin{equation}
\begin{split}
\frac{d\Re F(c+iy;k)}{dy}&=-\Im F'(s;k)|_{s=c+iy}\\
&=\sum_{j=0}^M\Im\psi(n+v_j+c+iy)-\sum_{j=1}^{M}\Im \psi(m_j+c+iy)
\end{split}	
\end{equation} 
to find that there exists some $\epsilon>0$ such that for all $|y|\geq n^{1/4}\rho_{M,n}^{1/4}(k)$,
\begin{equation}\label{esref(s;k)}
	\Re F(s;k)\geq \Re F(s_{\pm};k)+\epsilon n^{1/4}\rho_{M,n}^{1/4}(k)(|y|-n^{1/4}\rho_{M,n}^{1/4}(k)),
\end{equation}
where $s_{\pm}=1-k+(2+in^{1/4}\rho_{M,n}^{1/4}(k))\frac{1}{\rho_{M,n}(k)}$.
Combining \eqref{Ftestimate} and the value of $s_{\pm}$, we further have 
\begin{equation}\label{esf(s+-)}
	\Re \big(F(s_{\pm};k)-F(1-k;k)\big)=\frac{1}{2}n^{1/2}\rho_{M,n}^{1/2}(k)\big(1+o(1)\big).
\end{equation}
Using \eqref{esref(s;k)} and \eqref{esf(s+-)}, we have
\begin{equation}
\begin{split}
	I_1^{\text{global}}=&\frac{1}{\rho_{M,n}(k)}e^{-\frac{1}{2}n^{1/4}\rho_{M,n}^{1/4}(k)\big(1+o(1)\big)}\int_{\Re \sigma=2,|\Im \sigma|>n^{1/4}\rho_{M,n}^{1/4}(k)}\frac{d\sigma}{2\pi i}\oint_{|\tau|=1} \frac{d\tau}{2\pi i}\frac{\Gamma(t)}{\Gamma(s)}\frac{e^{\xi\tau}}{e^{\eta\sigma}}\frac{1}{\tau-\sigma}\\ \notag
	&O\big(e^{-\epsilon n^{1/2}\rho_{M,n}^{1/2}(k)(|\Im \sigma|-n^{1/4}\rho_{M,n}^{1/4}(k))}\big).\end{split}
\end{equation}
To estimate $I_2$ we need the following lemma, whose proof is left in the end of this
subsection.
\begin{lemma}\label{supergollemma}
There exists $\epsilon>0$, such that the inequality
\begin{equation}
	\Re\big(F(t;k)-F(1-k;k)\big)\leq -\epsilon\rho_{M,n}^2(k)|t-1+k|
\end{equation}
holds for all $t$ on $\Sigma_{-}(\frac{1}{2}-k) \cup \Sigma_{+}(\frac{3}{2}-k)$.
\end{lemma}
Together with the above lemma, combining the estimates \eqref{Ftestimate} for $s\in \mathcal{L}^{\text{local}}_{1-k}$, \eqref{esref(s;k)} and \eqref{esf(s+-)} for $s\in \mathcal{L}^{\text{global}}_{1-k}$, noting that $|s-t|>1/4$ when $\Delta_{M,n}$ is large, we have for some $\epsilon,\epsilon'>0$,
\begin{equation}\label{esI2}
	\begin{split}
	I_2=	&\int_{\mathcal{L}_{1-k}}\frac{ds}{2\pi i}\oint_{\Sigma_{-}(\frac{1}{2}-k) \cup \Sigma_{+}(\frac{3}{2}-k)}\frac{dt}{2\pi i}\frac{\Gamma(t)}{\Gamma(s)}	\frac{e^{(t+k-1)\rho_{M,n}(k) \xi}}{e^{(s+k-1)\rho_{M,n}(k) \eta}}
	\\ 
	&\times \begin{cases}
	\big(1+O(n^{-1/4}\rho_{M,n}^{-1/4}(k))\big)e^{-\frac{F''(1-k;k)(s-1+k)^2}{2}},& |\Im s|\leq n^{1/4}\rho_{M,n}^{-3/4}(k),\\
	e^{-\frac{1}{2}n^{1/2}\rho_{M,n}^{1/2}(k)-\epsilon'n^{-1/4}\rho_{M,n}^{3/4}(k)_(|\Im s|-n^{1/4}\rho_{M,n}^{-3/4}(k))},& \text{otherwise}.\\
	\end{cases}
	\end{split}
\end{equation}
And for $\xi,\eta$ in a compact set of $\mathbb{R}$, we have
\begin{equation}
	|e^{\xi\rho_{M,n}(k)(t+k-1)}/e^{\eta\rho_{M,n}(k)(s+k-1)}|\leq e^{C\rho_{M,n}(k)(|t+k-1|+|s+k-1|)}
\end{equation}
for some $C>0$.

Now  using the standard steepest-descent technique on $I_1^{\text{local}}$, $ I_1^{\text{global}}$ and $I_2$, we have that 
$|I_1^{\text{global}}|=o(e^{-n^{1/2}\rho_{M,n}^{1/2}(k)/4})$, $|I_2|=o(e^{-\epsilon\rho_{M,n}^2(k)/4})$.  Under the change of variables \eqref{cov}, when $n \to \infty$ and $ \Delta_{M,n} \to \infty$, we have
\begin{equation}\label{esI1}
I_1=\Big(1+O(\rho_{M,n}^{-1}(k))\Big)\frac{1}{\rho_{M,n}(k)}\int_{2-i\infty}^{2+i\infty}\frac{d\sigma}{2\pi i}\oint_{|\tau|=1}\frac{d\tau}{2\pi i}\frac{e^{-\frac{\tau^2}{2}+\xi\tau}}{e^{-\frac{\sigma^2}{2}+\eta\sigma}}\frac{\tau}{\sigma}\frac{1}{\sigma-\tau}.
\end{equation}
 From \eqref{esI2} and \eqref{esI1}, it is easy to see  that the integral $I_1$ concentrates on $I_1^{\text{local}}$ and $I_2$ is negligible relative to $I_1$. 

Combining $\lim\limits_{M+ n\to \infty}\rho_{M,n}(k) I_1=\frac{1}{\sqrt{2\pi}}e^{-\eta^2/2}$, we thus complete the proof of  Theorem \ref{supthm}.\end{proof}
\begin{proof}[Proof of Lemma \ref{supergollemma}]
	Assume $k$ is finite. When $n-k$ is finite, the proof follows by modifying the duality argument. For $t \in B(1-k, 1+k)$, we have
	\begin{equation}
		\rho_{M,n}^{-2}(k)\big(F(t;k) - F(1-k;k)\big) = -\frac{1}{2} \big(t - (1-k)^2\big) + O(n^{-1}),
	\end{equation}
	and $\frac{d}{dx} \Re F\big(-k \pm \frac{i}{4};k\big) > 0$. Consequently, along the segment $B(1-k,1+k) \cap \left[ \Sigma_{-}\big(\frac{1}{2}-k\big) \cup \Sigma_{+}\big(\frac{3}{2}-k\big) \right]$, the function $\rho_{M,n}^{-2}(k) \Re \big(F(t;k) - F(1-k;k)\big)$ decreases as $t$ moves away from $1-k$.
	
	Next, on the horizontal components of $\Sigma_{-}\big(\frac{1}{2}-k\big)$, 
	\begin{align*}
		\frac{d^2}{dx^2} \Re F\big(x \pm \tfrac{i}{4};k\big) 
		&= \Re \left. \frac{d^2F(t;k)}{dt^2} \right|_{t=x\pm\frac{i}{4}} \\
		&= -\left( \sum_{j=0}^M \Re \psi'\big(x+n+v_j \pm \tfrac{i}{4}\big) - \sum_{j=1}^M \Re \psi'\big(x+m_j \pm \tfrac{i}{4}\big) \right)< 0.
	\end{align*}
	Since $\frac{d}{dx} \Re F\big(x \pm \frac{i}{4};k\big) > 0$ at the endpoints of these horizontal contours, it follows that $\rho_{M,n}^{-2}(k)\Re F(t;k) $ decreases as $t$ moves away from $1-k$.
	Finally, for $t = -n + \frac{1}{2} + iy$ with $y \in [-\frac{1}{4}, \frac{1}{4}]$, a direct calculation yields
	\begin{equation}
		\Re \big(F(t;k) - F(1-k;k)\big) = -\rho_{M,n}^{2}(k) (n+1-k)^2 \big(1 + o(1)\big).
	\end{equation}

	Combining these estimates gives the desired result.
\end{proof}

\subsection{Proof of Theorem \ref{critthm}}
\begin{proof}[Proof of Theorem \ref{critthm}]

Applying the Euler's reflection formula for the gamma function
\begin{equation}
	\frac{\Gamma(t)\Gamma(1-t)}{\Gamma(s)\Gamma(1-s)}=\frac{\sin(\pi s)}{\sin(\pi t)},
\end{equation}
and the identity
\begin{equation}
\frac{\sin(\pi s)}{\sin(\pi t)}=\frac{\sin(\pi(s-t))}{\sin(\pi t)}e^{i \pi t}+e^{i \pi(t-s)},
\end{equation}
we rewrite $K_n(x,y)$ as
\begin{align}
K_n(x,y)&=\int_{\mathcal{C}}\frac{ds}{2\pi i}\oint_{\Sigma}\frac{dt}{2\pi i}\frac{e^{xt-ys}}{s-t}\frac{\Gamma(1-s)}{\Gamma(1-t)}\Big(\frac{\sin(\pi(s-t))}{\sin(\pi t)}e^{i \pi t}+e^{i \pi(t-s)}\Big)\notag\\
&\times\frac{\Gamma(s+n)}{\Gamma(t+n)}\prod_{j=1}^M\frac{\Gamma(s+n+v_j)\Gamma(t+m_j)}{\Gamma(t+n+v_j)\Gamma(s+m_j)}\notag\\
&=\int_{\mathcal{C}}\frac{ds}{2\pi i}\oint_{\Sigma}\frac{dt}{2\pi i}\frac{e^{xt-ys}}{s-t}\frac{\Gamma(1-s)}{\Gamma(1-t)}\frac{\sin(\pi(s-t))}{\sin(\pi t)}e^{i \pi t}\notag\\
&\times \frac{\Gamma(s+n)}{\Gamma(t+n)}\prod_{j=1}^M\frac{\Gamma(s+n+v_j)\Gamma(t+m_j)}{\Gamma(t+n+v_j)\Gamma(s+m_j)},
\end{align}
where the poles with respect to $t$ within $\Sigma$ are zeros of $\sin(\pi t)$.

After making change of variables $s\to s-[nu], t \to t-[nu]$ and 
 $x=g(\xi), y=g(\eta)$ in \eqref{cribgf}, we obtain
\begin{equation}
	K_n(g(\xi),g(\eta))=e^{\big(g(\eta)-g(\xi)\big)[nu]}\int_{\mathcal{C}}\frac{ds}{2\pi i}\oint_{\Sigma_{RE}}\frac{dt}{2\pi i}\frac{\sin \pi(s-t)}{s-t}\frac{e^{i \pi t}}{\sin \pi t} \frac{e^{f(\eta;s)}}{e^{f(\xi;t)}},
\end{equation}
where 
\begin{align}
	f(\xi;t)=&-tg(\xi)+\log\left(\frac{\Gamma(1+[nu]-t)}{\Gamma(1+[nu])}\right)\notag\\
	&+\sum_{j=0}^{M}\log\left(\frac{\Gamma(n-[nu]+t+v_j)}{\Gamma(n-[nu]+v_j)}\right)-\sum_{j=1}^{M}\log\left(\frac{\Gamma(m_j-[nu]+t)}{\Gamma(m_j-[nu])}\right),
\end{align}
and $\Sigma_{RE}$ is a rectangular contour with four points $\frac{1}{2}+[nu]-n\pm \frac{i}{2}$ and $\frac{1}{2}+[nu]\pm \frac{i}{2}$.

For $t \in B(0,n^{1/4})$ and $\xi$ in a compact subset of $\mathbb{R}$, we have uniformly
\begin{align}
	f(\xi;t)&=\frac{1}{2}\left(
	\sum_{j=0}^M
\frac{1}{n-[nu]+v_j}-\sum_{j=1}^M	\frac{1}{m_j-[nu]}+\frac{1}{1+[nu]}\right)t^2 \notag\\
&+\left(\sum_{j=0}^M\psi(n-[nu]+v_j)-\sum_{j=1}^M \psi(m_j-[nu])-\psi(1+[nu])-g(\xi)\right)t+O(n^{-1/4})\notag\\
&=\frac{1}{2}\big(\gamma'+o(1)\big)t^2-\xi t +O(n^{-1/4}),
\end{align}
where $\gamma'$ is defined in \eqref{gamma'}.\\
So as $s \in i\mathbb{R}$, $t\in \Sigma_{RE}$ and $s,t\in B(0,n^{1/4})$, $f(\xi;0)$ is a constant independent of $\xi$, we obtain
\begin{equation}
	\frac{e^{f(\eta;s)}}{e^{f(\xi;t)}}=\frac{e^{\frac{1}{2}\big(\gamma'+o(1)\big)s^2-\eta s}}{e^{\frac{1}{2}\big(\gamma'+o(1)\big)t^2-\xi t}}\Big(1+O(n^{-1/4})\Big).
\end{equation}

Next, we estimate the integrand when either $s$ or $t$ is not int $B(0,n^{1/4})$.
When $s=iy$, we have the following properties of $f(\eta;s)$, whose proof is left in the end of this
subsection.
\begin{lemma}\label{fsglobal}
	There exists a constant $\epsilon>0$ such that 
	\begin{equation}
		\Re f(\eta;iy)\leq -\epsilon n^{1/4}(|y|-n^{1/4}) \quad |y|>n^{1/4}.
	\end{equation}
\end{lemma}
By Lemma \ref{fsglobal}, it is easy to see that $\Re f(\eta;iy)$ dominates the factor $\sin \pi(s-t)$ as $s\to \infty$ on the verrtical contour.
On the other hand, as $t$ moves the right endpoint along $\{x\pm \frac{i}{2}\mid x\in [-n+[nu]+1/2,n^{1/4}]\}$ or to the left endpoint along $\{x\pm \frac{i}{2}\mid  x\in [n^{1/4}, [nu]+1/2]\}$, $\Re f(\xi;t)$ decrease monotonically. To see it, we check it on these horizontal contours. 
The second derivative 
\begin{align}
\frac{d^2\Re f(\xi;x\pm i/2)}{dx^2}
=	&\Re \Big(
\sum_{j=0}^M\psi'(n+v_j-[nu]+x\pm i/2)-\sum_{j=1}^M\psi'(mj-[nu]+x\pm i/2)\\ \notag
&+\psi'\big(1+[nu]-(x\pm i/2)\big)\Big)>0\end{align}	
and the first derivative 
\begin{equation}
	\left.\frac{d\Re f(\xi;x\pm i/2))}{dx}\right|_{x=n^{1/4}}>0,\quad \left.\frac{d\Re f(\xi;x\pm i/2))}{dx}\right|_{x=-n^{1/4}}<0,
	\end{equation}
so $\Re f(\xi;x \pm i/2)$ increases monotonically on these horizontal contours to the endpoints $\pm n^{1/4}\pm i/2$.
At last, on the two vertical lines of $\Sigma_{RE}$, applying the Stirling formula for $-\frac{1}{2}\leq y\leq \frac{1}{2}$, we have 
\begin{align}\label{valf-}
&	\Re f(\xi,-n+[nu]+\frac{1}{2}+iy)\\ \notag
&	=\sum_{j=1}^M (m_j-n)\log\big(\frac{m_j-nu}{m_j-n}\big)-v_j\log\big(\frac{n+v_j-nu}{v_j}\big)+O(n\log n),
\end{align}
and 
\begin{align}\label{valf+}
&\Re f(\xi,[nu]+\frac{1}{2}+iy)\\ \notag  
&=\sum_{j=1}^{M}m_j\Big(u+\log\big(\frac{m_j-nu}{m_j}\big)\Big)-\sum_{j=0}^{M}(n+v_j)\Big(u+\log\big(\frac{n+v_j-nu}{n+v_j}\big)\Big)+O(n\log n).
\end{align}

Togeter with  \eqref{valf+}, \eqref{valf-} and Lemma \ref{fsglobal}, we see that
the double integral concentrates on the origin of $s,t \in B(0,n^{1/4})$, that is 
\begin{align}\label{crilocal}
&	e^{\big(g(\xi)-g(\eta)\big)[nu]}K_n(g(\xi),g(\eta))=(1+O(n^{-1/4})) \notag\\
&\times	\big(\int_{-in^{1/4}}^{in^{1/4}}\frac{ds}{2\pi i}\int_{-i/2-n^{1/4}}^{-i/2+n^{1/4}}\frac{dt}{2\pi i}+\int_{-in^{1/4}}^{in^{1/4}}\frac{ds}{2\pi i}\int_{i/2+n^{1/4}}^{i/2-n^{1/4}}\frac{dt}{2\pi i}\big)\\ \notag
&\times \frac{\sin \pi(s-t)}{s-t}\frac{e^{i \pi t}}{\sin \pi t}\frac{e^{\frac{\gamma'}{2}s^2-\eta s}}{e^{\frac{\gamma'}{2}t^2-\xi t}}.
\end{align}
Substituting the identity
\begin{equation}
	\frac{\sin\pi(s-t)}{s-t}=\frac{\pi}{2}\int_{-1}^1dwe^{i\pi(t-s)w}
\end{equation}
into the \eqref{crilocal} and as $n\to \infty$, this gives us 
\begin{align}
&\lim_{n \to \infty}e^{\big(g(\xi)-g(\eta)\big)[nu]}K_n(g(\xi),g(\eta))\notag\\
&=\sqrt{\frac{\pi}{8\gamma'}}\int_{-1}^1 e^{\frac{(\pi w-i\eta)^2}{2\gamma'}}\big(\int_{-i/2-\infty}^{-i/2+\infty}-\int_{i/2-\infty}^{i/2+\infty}\big)\frac{dt}{2\pi i}\frac{e^{i \pi t}}{\sin \pi t}e^{-\frac{\gamma'}{2}t^2+\xi t +i\pi wt}dw\\
&=\sqrt{\frac{\pi}{8\gamma'}}\int_{-1}^1 e^{\frac{(\pi w-i\eta)^2}{2\gamma'}}\frac{1}{\pi}\theta\big(\frac{\pi w-i\xi}{2\pi},\frac{i\gamma'}{2\pi}\big)dw,
\end{align}
and the last equality follows form the residue theorem to evaluate the inner integral with respect to $t$, where the poles are $z\in \mathbb{Z}$.

The proof of $u=0$ is similar to the  proof of $u\in(0,1)$, only the  scaling function $g(\cdot)$ and corresponding parameters need to be substituted, 
so we omit it. 

We thus complete the proof of Theorem\ref{critthm}.\end{proof}
\begin{proof}[Proof of Lemma \ref{fsglobal}]
	We begin by computing the derivative of $\Re f(\eta; iy)$ with respect to $y$:
	\begin{equation}
		\frac{d}{dy}\Re f(\eta;iy) = -\Im \left. \frac{df(\eta;z)}{dz} \right|_{z=iy}.
	\end{equation}
	Substituting the given expression for the derivative, we obtain
	\begin{align}
		\frac{d}{dy}\Re f(\eta;iy) = \Im \bigg( \psi(1 + [nu] - iy) + \sum_{j=1}^M \psi(m_j - [nu] + iy)- \sum_{j=0}^M \psi(n - [nu] + v_j + iy) \bigg).
	\end{align}
	This simplifies to
	\begin{align}
	&	\frac{d}{dy}\Re f(\eta;iy)= -\left(1 + O(n^{-1})\right) \notag\\
		&\times \Bigg( \arctan\left(\frac{y}{1 + [nu]}\right) + \sum_{j=0}^M \arctan\left(\frac{y}{n - [nu] + v_j}\right)
 - \sum_{j=1}^M \arctan\left(\frac{y}{m_j - [nu]}\right) \Bigg).
	\end{align}
	For $|y| > n^{1/4}$, the dominant behavior of the arctangent terms ensures the desired result.
\end{proof}

\subsection{Proof of Theorem \ref{subthm} $\mathtt{(i)}$}
\begin{proof}[Proof of Theorem \ref{subthm} $\mathtt{(i)}$]

After changing the variables $s,t \to \rho_{M,n}^{3/2}s,\rho_{M,n}^{3/2}t$, we obtain
\begin{equation}
K_n(x,y)=\rho_{M,n}^{3/2}\int_{\mathcal{C}}\frac{ds}{2 \pi i} \oint_{\Sigma}\frac{dt}{2\pi i}\frac{e^{\rho_{M,n}^{3/2}(xt- ys)}}{s-t}
	\prod_{j=0}^{M}\frac{\Gamma(\rho_{M,n}^{3/2}s+n+v_j)}{\Gamma(\rho_{M,n}^{3/2}t+n+v_j)}\frac{\Gamma(\rho_{M,n}^{3/2}t+m_j)}{\Gamma(\rho_{M,n}^{3/2}s+m_j)}.
\end{equation}
After changing $x=g(\xi),y=g(\eta)$ in \eqref{subcgd}, we have 
\begin{equation}
K_n(g(\xi),g(\eta))=\rho_{M,n}^{3/2}\int_{\mathcal{C}}\frac{ds}{2 \pi i} \oint_{\Sigma}\frac{dt}{2\pi i}\frac{e^{\rho_{M,n}^{1/2}(\xi t-\eta s)}}{s-t}e^{F_{M,n}(\rho_{M,n}^{3/2}s)-F_{M,n}(\rho_{M,n}^{3/2}t)},
\end{equation}
where 
\begin{equation}
	F_{M,n}(s)=\sum_{j=0}^M\log \frac{\Gamma(s+n+v_j)}{\Gamma(n+v_j)}-\sum_{j=1}^M \log \frac{\Gamma(s+m_j)}{\Gamma(m_j)}-\log\Gamma(s)-s\log(\lambda_M).
\end{equation}
Using  the stirling formula for $\log \Gamma(z)$ in \eqref{eslogg},
we have 
\begin{equation}
F_{M,n}(\rho_{M,n}^{3/2}s)=\rho_{M,n}^{3/2}f_{M,n}(s)+O(\Delta_{M,n}),
\end{equation}
where 
\begin{align}
	f_{M,n}(s)=&\sum_{j=0}^M\frac{n+v_j}{\rho_{M,n}^{3/2}}\Big(1+\frac{\rho_{M,n}^{3/2}s}{n+v_j}\Big)\log \Big(1+\frac{\rho_{M,n}^{3/2}s}{n+v_j}\Big)-\sum_{j=1}^M \frac{m_j}{\rho_{M,n}^{3/2}}\Big(1+\frac{\rho_{M,n}^{3/2}s}{m_j}\Big)\log \Big(1+\frac{\rho_{M,n}^{3/2}s}{m_j}\Big)\\ \notag
	&-s\log s
	+s\Big(\sum_{j=0}^M\log(n+v_j)-\sum_{j=1}^M\log m_j-\log\rho_{M,n}^{3/2}-\log\lambda_M\Big).
\end{align}
Set $v_{M,n}=\sum_{j=0}^M\log(n+v_j)-\sum_{j=1}^M\log m_j-\log\rho_{M,n}^{3/2}-\log \lambda_M$, then we have 
\begin{equation}
	f_{M,n}'(s)=\sum_{j=0}^M \log\Big(1+\frac{\rho_{M,n}^{3/2}s}{n+v_j}\Big)-\sum_{j=1}^M \log \Big(1+\frac{\rho_{M,n}^{3/2}s}{m_j}\Big)-\log s +v_{M,n},
\end{equation}
and 
\begin{equation}
f_{M,n}''(s)=\rho_{M,n}^{3/2}\sum_{j=0}^M\left(\frac{1}{n+v_j+\rho_{M,n}^{3/2}s}-\frac{1}{m_j+\rho_{M,n}^{3/2}s}\right),
\end{equation}
\begin{equation}
	f_{M,n}'''(s)=\rho_{M,n}^3\sum_{j=0}^M\left(\frac{1}{(m_j+\rho_{M,n}^{3/2}s)^2}-\frac{1}{(n+v_j+\rho_{M,n}^{3/2}s)^2}\right).
\end{equation}
On the other hand, we recall $z_0$ is given in \eqref{defz_0}
and note that $q_0:=\frac{z_0}{\rho_{M,n}^{3/2}}$, thus we have 
\begin{equation}
	f_{M,n}'(q_0)=0,\quad 	f_{M,n}''(q_0)=0,
\end{equation}
and \begin{equation}
	f_{M,n}'''(q_0)=2+o(1).
\end{equation}

Next in the double contour integral formula for $K_n(x,y)$, we deform the contour for $s$ as $\mathcal{C}=\mathcal{C}_{\text{local}} \cup \mathcal{C}_{\text{global}}$, where  $\mathcal{C}_{\text{local}}=\mathcal{C}_{\text{local},+} \cup \mathcal{C}_{\text{local},-}$, such that 
\begin{align}
&\mathcal{C}_{\text{local},+}=\Big\{q_0+\big(1/\rho_{M,n}^{3/2}\big)^{1/3}+re^{\frac{4\pi i}{3}}\mid -\big(1/\rho_{M,n}^{3/2}\big)^{3/10}\leq r \leq 0\Big\},\\
&\mathcal{C}_{\text{local},-}=\Big\{q_0+\big(1/\rho_{M,n}^{3/2}\big)^{1/3}+re^{\frac{5\pi i}{3}}\mid 0\leq r \leq \big(1/\rho_{M,n}^{3/2}\big)^{3/10}\Big\},
\end{align}
and 
\begin{align}
\mathcal{C}_{\text{global}}=&\Big\{q_0+\big(1/\rho_{M,n}^{3/2}\big)^{1/3}-\big(1/\rho_{M,n}^{3/2}\big)^{1/3}e^{\frac{4 \pi i}{3}}+iy \mid y\geq 0\Big\}\notag \\
&\cup \Big\{q_0+\big(1/\rho_{M,n}^{3/2}\big)^{1/3}+\big(1/\rho_{M,n}^{3/2}\big)^{1/3}e^{\frac{5 \pi i}{3}}+iy \mid y\leq 0\Big\}.
\end{align}
Deform the contour $\Sigma$ for $t$ as $\Sigma_{\text{local}} \cup \Sigma_{\text{global}} $, such that  $\Sigma_{\text{local}}=\Sigma_{\text{local},+} \cup \Sigma_{\text{local},-}$ where 
\begin{align}
	&\Sigma_{\text{local},+}=\Big\{q_0-\big(1/\rho_{M,n}^{3/2}\big)^{1/3}+re^{\frac{2\pi i}{3}}\mid 0\leq r \leq \big(1/\rho_{M,n}^{3/2}\big)^{3/10}\Big\},\\
	&\Sigma_{\text{local},-}=\Big\{q_0-\big(1/\rho_{M,n}^{3/2}\big)^{1/3}+re^{\frac{\pi i}{3}}\mid  -\big(1/\rho_{M,n}^{3/2}\big)^{3/10}\leq r \leq 0\Big\}.
\end{align}
By constructing the integral contour $\Sigma_{\text{global}}$, we can prove that $\Re f_{M,n}(t) $ attains its minimum at left end of $\Sigma_{\text{local},\pm}$. We define 
\begin{equation}
\Sigma_{\text{global}}:=\Sigma_{+}^1 \cup \Sigma_{+}^2 \cup \Sigma_{+}^3 \cup \Sigma^4 \cup  \Sigma_{-}^1 \cup \Sigma_{-}^2 \cup \Sigma_{-}^3,
\end{equation}
where 
\begin{align}
	&\Sigma_{\pm}^1= \Big\{\text{vertical line connecting the left end of $\Sigma_{\text{local},\pm}$ to $(x_1,\pm y_1)$}\Big\}, \notag\\
	&\Sigma_{\pm}^2=\Big\{t=x \pm iy \mid \frac{y-y_1}{y_2-y_1}=\frac{x-x_1}{x_2-x_1} \Big\},\notag\\
	&\Sigma_{\pm}^3=\Big\{t=x \pm iy_2 \mid x\in[-n/\rho_{M,n}^{3/2}+(2\rho_{M,n}^{3/2})^{-1},x_2]\Big\},\notag\\
	&\Sigma^4=\Big\{t=-n/\rho_{M,n}^{3/2}+(2\rho_{M,n}^{3/2})^{-1}+iy \mid y \in [-y_2,y_2]\Big\}.
\end{align}
The orientation of the contours defined as above are detemined by $\Sigma$ is positively oriented, and $\mathcal{C}$ is from $-i \infty$ to $i \infty$.

In the following, we provide the selection strategy for the contour $\Sigma_{\text{global}}$.
For $t \in \Sigma_{\pm}^1$, we set \begin{equation}
x_1=q_0-\big(\rho_{M,n}^{-3/2}\big)^{1/3}-\frac{1}{2}\big(\rho_{M,n}^{-3/2}\big)^{3/10},
\end{equation} and choose $y_1$ satisfies
$\frac{d\Re f_{M,n}(x_1+iy_1)}{dy}=0$.
Next, we choose  $C$ as a large enough positive constant, independent of $M,n$ and let $x_2=-C$, such that  $\frac{d\Re f_{M,n}(x_2+iy)}{dx}<0$ for $t\in \Sigma^3_{+}$ and choose $y_2$ such that $\frac{d \Re f_{M,n}(x_2+iy_2)}{dy}=0$.
At last, we illustrate that $\Re f_{M,n}(t)$ increases as $t$ along the $(x_1,y_1)\to (x_2,y_2)$.
Note that $z=re^{i\theta}$, $r>0$ and  the angle $\theta \in (\pi/2,\pi)$ is fixed.

 By the first derivative 
\begin{equation}
	\begin{split}\frac{d\Re f_{M,n}(x_1+iy_1+z)}{dr}&=\Re\big(e^{i\theta}f'_{M,n}(x_1+iy_1+z)\big)\\
		&=\cos\theta \Re f'_{M,n}(x_1+iy_1+z)-\sin \theta \Im f'_{M,n}(x_1+iy_1+z),
	\end{split}
\end{equation}
and second derivative 
\begin{equation}\label{urr}
	\frac{d^2\Re f_{M,n}(x_1+iy_1+z)}{dr^2}=-\frac{1}{r}\frac{d\Re f_{M,n}(x_1+iy_1+z)}{dr}-\frac{1}{r^2}	\frac{d^2\Re f_{M,n}(x_1+iy_1+z)}{d\theta^2},
\end{equation}
we see that the signs of first-order and second-order derivatives are opposite. Combining \eqref{urr} and  $\frac{d\Re f_{M,n}(x_1+iy_1)}{dr}>0$ and $\frac{d\Re f_{M,n}(x_2+iy_2)}{dr}>0$, we can  conclude that $\frac{d\Re f_{M,n}(t)}{dr}>0$ for $t\in \Sigma^2_{+}$.
Similarly, we can prove that 
$\Re f_{M,n}(t)$ increases as $t$ moves upward on $\Sigma^1_{+}$, or t moves downward on $\Sigma^1_{-}$ and  $\Re f_{M,n}$ increases as $t$ moves leftward, and attains its minimum at $(x_2,\pm y_2)$ for $t\in \Sigma^3_{\pm}$.  It is easy to see that $\Re f_{M,n}(t)$ attains its maximun at $t=-n/\rho_{M,n}^{3/2}+(2\rho_{M,n}^{3/2})^{-1}$ for $t \in \Sigma^4$.
So  $\Re f_{M,n}(t) $ attains its minimum at left end of $\Sigma_{\text{local},\pm}$ on $\Sigma_{\text{global}}$.

 Makig the change of variables \begin{equation}
t=q_0+\rho_{M,n}^{-1/2}\tau, \quad s=q_0+\rho_{M,n}^{-1/2}\sigma,
 \end{equation}
 and taking Taylor expansion at $q_0$
 in $B(q_0,\rho_{M,n}^{-\frac{3}{2}\times \frac{3}{10}})$, we have
\begin{equation}
	f_{M,n}(t)=f_{M,n}(q_0)+\frac{1}{6}f_{M,n}'''(q_0)(t-q_0)^3+O(\rho_{M,n}^{-\frac{3}{2}\times \frac{6}{5}}),
\end{equation}
and 
\begin{align}\label{localestimate}
	&\big(\rho_{M,n}^{1/2}\big)e^{-\rho_{M,n}^{1/2}(\xi-\eta)q_0}\int_{\mathcal{C}_{\text{local}}}\frac{ds}{2\pi i}\int_{\Sigma_{\text{local}}}\frac{dt}{2 \pi i}e^{\rho_{M,n}^{1/2}(\xi t -\eta s)}\frac{1}{s-t}e^{ \rho_{M,n}^{3/2}\frac{f_{M,n}'''(q)}{6}\big((s-q)^3-(t-q)^3\big)+O(\rho_{M,n}^{-\frac{3}{10}})}\notag \\
	&=\int_{\mathcal{C}^{\rho_{M,n}^{1/20}}_{<}}\frac{d\sigma}{2\pi i} \int_{\Sigma_{>}^{\rho_{M,n}^{1/20}}}\frac{d \tau}{2\pi i }\frac{1}{\sigma-\tau}\frac{e^{\frac{\sigma^3}{3}-\eta \sigma}}{e^{\frac{\tau^3}{3}-\xi \tau}}\big(1+O(\rho_{M,n}^{-\frac{3}{10}})\big) \\ \notag
	&=K_{\mathtt{Ai}}(\xi,\eta)+O(\rho_{M,n}^{-\frac{3}{10}}).
\end{align}
Here, for real $R>0$ the contours are defined to be upwards and parameterized as 
\begin{align*}
	&\mathcal{C}_{<}^R=\{1+re^{\frac{4\pi i}{3}}\mid -R\leq r\leq 0\} \cup \{1+re^{\frac{5\pi i}{3}}\mid 0\leq r\leq R\},\\
	&\Sigma_{>}^R=\{-1+re^{\frac{2\pi i}{3}}\mid 0\leq r\leq R\} \cup \{-1+re^{\frac{\pi i}{3}}\mid -R\leq r\leq 0\}.
\end{align*}

On the other hand, we have the ``global" estimates of $\Re f_{M,n}(s)$ for $s\in \mathcal{C}_{\text{global}}$,
and  of $\Re f_{M,n}(t)$ for $t \in \Sigma_{\text{global}}$.
For $s$, we consider the infinite long contour $\mathcal{C}_{\text{global}}$ in two parts, one is $\mathcal{C}_{\text{global}}^1=\{s\in \mathcal{C}_{\text{global}} \mid |\Im s|\leq K\}$ and $\mathcal{C}_{\text{global}}^2=\mathcal{C}_{\text{global}}\setminus \mathcal{C}_{\text{global}}^1$, where $K$ is a large positive constant. 
\begin{lemma}\label{globalref}
	Let $x_0>q_0$, and  $C=\{x_0+iy\mid y\in \mathbb{R}\}\}$, then $\Re f_{M,n}(s)$ attains its global maximum at $x_0$ on $C$. Moreover, 
	\begin{equation}
		\frac{d \Re f_{M,n}(x_0+iy)}{dy}\begin{cases}
			<0, &y>0,\\
			>0,  & y<0.\\
		\end{cases}
	\end{equation}
\end{lemma}
By Lemma \ref{globalref}, $\Re f_{M,n}(s)$ decreases as $s$ moves upward on $\mathcal{C}_{\text{global}}^1 \cap \mathbb{C}_{+}$ or $s$ moves downward on $\mathcal{C}_{\text{global}}^1 \cap \mathbb{C}_{-}$.
We can also show by direct computation that $\Im f'_{M,n}(t)>1$ if $s\in \mathcal{C}_{\text{global}}^2 \cap \mathbb{C}_{+}$ and $\Im f'_{M,n}(t)<-1$ if $s\in \mathcal{C}_{\text{global}}^2 \cap \mathbb{C}_{-}$. Hence $f_{M,n}(s)$ decreases at least linearly fast as $s$ moves to $\pm i \infty$ along $\mathcal{C}_{\text{global}}^2$.
For $t \in \Sigma_{\text{global}}$,  we know that $\Re f_{M,n}(t;\lambda_M)$ attains a minimum at the leftend of $\Sigma_{\text{local},\pm}$.
Combining the above arguments of Global estimates for $s\in \mathcal{C}_{\text{global}}$ and $t \in \Sigma_{\text{global}}$, there exists $\epsilon>0$ such that for $t\in \Sigma_{\text{global}}$ and $s \in \mathcal{C}_{\text{global}}$,
\begin{equation}\label{subglobales}
\Re \Big(f_{M,n}(s)-\frac{\xi(t-q_0)}{\rho_{M,n}}\Big)+\epsilon \rho_{M,n}^{-27/20}\\
	\leq \Re \Big(f_{M,n}(t)-\frac{\eta (t-q_0)}{\rho_{M,n}}\Big)-\epsilon \rho_{M,n}^{-27/20}.
\end{equation}
Hence combining the
estimates \eqref{subglobales} and  \eqref{localestimate} of $f_{M,n} (t; \lambda_M)$,  we have 
\begin{equation}\label{globalestimate}
	\begin{split}
	&	\rho_{M,n}^{-1}e^{-(\xi-\eta)\frac{z_0}{\rho_{M,n}}}\int\int_{\mathcal{C}\times\Sigma\setminus \mathcal{C}_{\text{local}}\times \Sigma_{\text{local}}}\frac{ds}{2\pi i}\frac{dt}{2\pi i}\frac{1}{s-t}e^{\rho_{M,n}^{1/2}(\xi t-\eta s)}e^{\rho_{M,n}^{3/2}(f_{M,n}(s)-f_{M,n}(t))}\\
	&=\rho_{M,n}^{-1}\int\int_{\mathcal{C}\times\Sigma\setminus \mathcal{C}_{\text{local}}\times \Sigma_{\text{local}}}\frac{ds}{2\pi i}\frac{dt}{2\pi i}\frac{1}{s-t}
	\frac{e^{\rho_{M,n}^{3/2}\Big(f_{M,n}(s)-\frac{\eta(s-q_0)}{\rho_{M,n}}\Big)}}{e^{\rho_{M,n}^{3/2}\Big(f_{M,n}(t)-\frac{\xi(t-q_0)}{\rho_{M,n}}\Big)}}\\
	&=O\big(e^{-\epsilon(\rho_{M,n}^{-3/20})}\big).
	\end{split}
\end{equation}

Combining the ``local" estimate \eqref{localestimate} and ``global" estimate \eqref{globalestimate}, we thus complete the proof of  Theorem \ref{subthm}.
\end{proof}
	
	\begin{proof}[Proof of Lemma \ref{globalref}]
		We analyze the derivative of the real part of $f_{M,n}$ along the imaginary axis
		\begin{align}
			\frac{d}{dy}\Re f_{M,n}(x_0 + iy) 
			&= -\Im \left. \frac{df_{M,n}(z)}{dz} \right|_{z=x_0+iy} \notag\\
			&= \Im \bigg( \log z + \sum_{j=1}^M \log\Big(1+\frac{\rho_{M,n}^{3/2}z}{m_j}\Big) 
			- \sum_{j=0}^M \log\Big(1+\frac{\rho_{M,n}^{3/2}z}{n+v_j}\Big) \bigg) \bigg|_{z=x_0+iy} \notag\\
			&= \sum_{j=0}^M \left( \arctan\left(\frac{\rho_{M,n}^{3/2}y}{m_j + \rho_{M,n}^{3/2}x_0}\right) 
			- \arctan\left(\frac{\rho_{M,n}^{3/2}y}{n+v_j + \rho_{M,n}^{3/2}x_0}\right) \right)  \notag\\
			&= -\rho_{M,n}^{3/2}y \sum_{j=0}^M \left( \frac{1}{n+v_j + \rho_{M,n}^{3/2}x_0} 
			- \frac{1}{m_j + \rho_{M,n}^{3/2}x_0} \right) + O(\Delta_{M,n})
		\end{align}
		The key observation comes when $\rho_{M,n}^{3/2}x_0 > z_0$, where $z_0$ is the unique positive solution to equation \eqref{defz_0}, we have
	\begin{equation}
				\sum_{j=0}^M \left( \frac{1}{n+v_j + \rho_{M,n}^{3/2}x_0} - \frac{1}{m_j + \rho_{M,n}^{3/2}x_0} \right) > 0.
	\end{equation}
	This inequality implies $\frac{d}{dy}\Re f_{M,n}(x_0+iy) < 0$ for $y > 0$
and $\frac{d}{dy}\Re f_{M,n}(x_0+iy) > 0$ for $y < 0$.

		Consequently, $\Re f_{M,n}(s)$ attains its global maximum at $x_0$ on the contour $C$.
	\end{proof}

\section{Bulk statistics in Low DWR}\label{sect3}
In this section, we investigate a simplified version of the  {\bf{Truncated  product model}} under \eqref{assumptionss}. This  enables an exact  formulation of the limiting density of squared singular values for $Y_M$, expressed through the following  parameterization.


Let $G(z)$ be the Stielties transform of the probability measure of the squared singular values of $Y_M$, which satisfies the algebraic equation \begin{equation}
	z(zG(z)-1)\Big(zG(z)+\frac{1}{a}\Big)^M=(zG(z))^{M+1}.
\end{equation}
By making change of variable $W(z)=zG(z)$, we obtain 
\begin{equation}	W^{M+1}=(W-1)\Big(W+\frac{1}{a}\Big)^M z,\end{equation}
Let $\frac{(a+1)W}{aW+1}=V$, then
\begin{equation}\label{aglequ}	V^{M+1}=(a+1)^{M+1}a^{-M}(V-1)z.\end{equation}
According to the parameterization of denstiy function of  Fuss-Catalan distribution\cite{Ne14}, we set
\begin{equation}\label{x_para}	x=x(\theta)=\frac{a^M(\sin((M+1)\theta))^{M+1}}{(a+1)^{M+1}\sin(\theta)(\sin(M\theta))^M}, \quad 0<\theta<\frac{\pi}{M+1},
\end{equation}
then \begin{equation}	V_{+}(x)=\frac{\sin((M+1)\theta)}{\sin(M\theta)}e^{i\theta},\quad	V_{-}(x)=\frac{\sin((M+1)\theta)}{\sin(M\theta)}e^{-i\theta}
\end{equation}
are two solutions of \eqref{aglequ}.
So we  get the formula of the density function  
\begin{align}
\label{den_para}	\rho(\theta)&=\frac{1}{2\pi i x}(W_{+}(x)-W_{-}(x))
		=\frac{1}{\pi x} \Im\big(\frac{V_+(x)}{a+1-aV_+(x)}\big)\\
	&=\frac{(a+1)^{M+2}\sin^{M+1}(M\theta)\sin^2(\theta)}{\pi a^M\sin^{M}((M+1)\theta)\big((a+1)^2\sin^2(M\theta)\sin^2(\theta)+(a\cos(M\theta)\sin(\theta)-\sin(M\theta)\cos(\theta))^2\big)}\notag,
\end{align}
supported on $\big[0,\frac{(M+1)^{M+1}a^M}{(a+1)^{M+1}M^{M}}\big]$,
where \begin{equation}
	W_{\pm}(x)=\frac{V_{\pm}(x)}{a+1-aV_{\pm}(x)}=\frac{\sin((M+1)\theta)e^{\pm i\theta}}{(a+1)\sin (M\theta)-a\sin((M+1)\theta)e^{\pm i\theta}}.
\end{equation}
In the case $a=1,M=1$, the density reduce to the well-known arcsine measure. 



 When $0<a<M$,
the squared singular values of $Y_M$ are a determiantal point process\cite{CKW15} with correlation kernel 
\begin{equation}\label{k_n}
	\widetilde{K}_n(x,y)=\frac{1}{(2 \pi i)^2}\int_{\mathcal{C}}ds \oint_{\Sigma}dt \prod_{j=0}^{M}\frac{\Gamma(s+1+v_j)}{\Gamma(t+1+v_j)}\frac{\Gamma(t+1-n+m_j)}{\Gamma(s+1-n+m_j)}\frac{x^ty^{-s-1}}{s-t}.
\end{equation}
Using the argument of parametrization, we can establish bulk universality in Low DWR regime stated in Theorem \ref{TBU} $\mathtt{(ii)}$.
And the following theorem is an exact expression of Theorem \ref{subthm} $\mathtt{(i)}$ for the kernel $\widetilde{K}_n$.
\begin{theorem}\label{TSEU}{\bf(GUE edge statistics)}
 Under \eqref{assumptionss}, suppose that $\lim\limits_{M+n \to \infty} \Delta_{M,n} = 0$.
Let
\begin{equation}\label{defxright}
	x_{*}=\frac{(M+1)^{M+1}a^M}{(a+1)^{M+1}M^M},\quad c_2=2^{-\frac{1}{3}}\frac{a^M(M+1)^{M+2/3}(M-a)^{4/3}}{(a+1)^{M+5/3}M^{M+1/3}},
\end{equation} 
then uniformly for $\xi$ and $\eta$ in any compact subset of $\mathbb{R}$, we have
	\begin{equation}
		\lim_{n \to \infty}e^{n^{1/3}(\eta-\xi)\frac{M+1}{M-a}\frac{c_2}{x_*}}n^{-2/3}c_2\widetilde{K}_n\big(x_{*}+\frac{c_2\xi}{n^{2/3}},x_{*}+\frac{c_2\eta}{n^{2/3}}\big)=K_{\mathtt{Ai}}(\xi,\eta).
	\end{equation}
 
\end{theorem}
In the following,  we provide proofs for the bulk and edge limit theorems in a unified treatment through the kernel $\widetilde{K}_n$.


\subsection{Proof of Theorem \ref{TBU}  $\mathtt{(ii)}$}
Although \cite{LWZ16} addresses finite products and \cite{LWW23}  deals with  infinite products, our proof technique extends the framework of \cite{LWZ16} to both cases in the bulk. We outline the proof strategy by partitioning $\Sigma$ into curved and vertical components, denoted respectively as $\Sigma_{curved}$ and $\Sigma_{vertical}$. The kernel decomposes as $\widetilde{K}_n(x,y) = I_1 + I_2$, where:
\begin{itemize}
	\item $I_1$, which is defined over $\mathcal{C} \times \Sigma_{curved}$, is evaluated via saddle point analysis and shown to be negligible;
	\item $I_2$, which is defined over ${\mathcal{C} \times \Sigma_{vertical}}$, is computed through Cauchy's theorem and constitutes the principal distribution.
\end{itemize}

Let us first introduce some  notations. 
The integral contour for $s$ denoted by $\mathcal{C}$ is the vertical, upward contour through the points $nW_+$ and $nW_-$.
  The contour for $t$ denoted by $\Sigma$ is defined as 
\begin{equation}\Sigma=\Sigma_{curved}\cup \Sigma_{vertical},
\end{equation}
where
\begin{equation}\Sigma_{curved}=\Sigma_1 \cup \Sigma_2,\quad \Sigma_{vertical}=\Sigma_3 \cup \Sigma_4,
\end{equation}
and  \begin{equation}
	\begin{split}
&\Sigma_1=n\tilde{\Sigma}^{r}\cap \{z \mid \Re z \leq \Re nW_{\pm}-\epsilon\},\\
&\Sigma_2=n\tilde{\Sigma}^{r}\cap \{z \mid \Re z \geq \Re nW_{\pm}+\epsilon\},\\
&\Sigma_3={\text{vertical bar connecting the two ending points of $\Sigma_1$}},\\
&\Sigma_4={\text{vertical bar connecting the two ending points of $\Sigma_2$}},
\end{split}
\end{equation}
where $\epsilon$ and $\delta$ are two small enough positive constant and we take $r=\frac{1}{n}\big([\delta n]+\frac{1}{2}\big)$.
For the construction of $\tilde{\Sigma}^r$, we first define $\tilde{\Sigma}:=\tilde{\Sigma}_+ \cup \tilde{\Sigma}_-$,
where
\begin{equation}
\tilde{\Sigma}_+ :=\{\zeta(\theta) \mid \theta \in [0,\frac{\pi}{M+1}]\},\quad \tilde{\Sigma}_+ :=\{\overline{\zeta(\theta)} \mid \theta \in [0,\frac{\pi}{M+1}]\},
\end{equation}
with
\begin{equation}
	\zeta(\theta)=\frac{\sin((M+1)\theta)e^{i\theta}}{(a+1)\sin (M\theta)-a\sin((M+1)\theta)e^{ i\theta}}.
\end{equation}
It is easy to check that $\tilde{\Sigma}_{\pm}$ lies in $\mathbb{C}_{\pm}$, pass through $nW_{\pm}$ and intersects the real line only at $0$ when $\theta= \frac{\pi}{M+1}$ and at $\frac{M+1}{M-a}$ when $\theta=0$. Furthermore, as $\theta$ runs from $0$ to $\frac{\pi}{M+1}$, the value of $|\zeta(\theta)|$ decreases.
And we need to make a small deformation of $\tilde{\Sigma}$ around the origin, which is given by 
\begin{equation*}
	\tilde{\Sigma}^r:=\Big\{z \in \tilde{\Sigma} \mid |z|\geq r \Big\}\cup \Big\{\text{the arc of $\{|z|=r\}$ connecting $\tilde{\Sigma}\cap \{|z|=r\}$ \text{and through $-r$}} \Big\}.
\end{equation*}
\begin{proof}[Proof of  Theorem \ref{TBU} $\mathtt{(ii)}$]
For asymptotics  of $\tilde{K}_n(x,y)$, we rewrite \eqref{k_n} as
\begin{equation}
\tilde{K}_n(x,y)=I_1+I_2,
\end{equation}
where
 \begin{equation}
	I_1=\frac{1}{y(2\pi i )^2} \int_{\mathcal{C}}\oint_{\Sigma_{curved}} \frac{e^{F(s;y)}}{e^{F(t;x)}}\frac{1}{s-t}dsdt,
\end{equation} 
and  
\begin{equation}
	I_2=\frac{1}{y(2\pi i )^2} \int_{\mathcal{C}}\oint_{\Sigma_{vertical}} \frac{e^{F(s;y)}}{e^{F(t;x)}}\frac{1}{s-t}dsdt,
	\end{equation}
where
\begin{equation}
	F(z;x)=t\log x+\sum_{j=0}^M\log \Gamma(z+1+v_j)-\log \Gamma(z+1-n+m_j).
\end{equation}
Using the Stirling's formula for gamma function, we have
 \begin{equation}\label{FBS}
	F(z;x)
	=n\hat{F}(z/n;x)+(1-M/a)\log n +c_{M,n}+O(\Delta_{M,n}),
\end{equation}
where \begin{equation}\label{hatF}
	\begin{split}	\hat{F}(z;x)&=(M+1)z\big(\log z-1\big)-(z-1)\big(\log(z-1)-1\big)\\
		&	-M(z+\frac{1}{a})\big(\log(z+\frac{1}{a})-1\big)-z\log x,
\end{split}	\end{equation}
and 
\begin{equation}\label{C-mn}
	c_{M,n}=\sum_{j=0}^M\Big(v_j+\frac{1}{2}\Big)\log z-\frac{1}{2}\log (z-n)-\sum_{j=1}^M\Big(v_j+\frac{1}{2}\Big)\log\Big(z+\frac{n}{a}\Big).
\end{equation}
For $x_0\in \big(0,\frac{(M+1)^{M+1}a^M}{(a+1)^{M+1}M^M}\big)$, we use the paramentrization shown in \eqref{x_para}, and let $\theta$ be the unique real number in $(0,\frac{\pi}{M+1})$
such that \eqref{x_para} and \eqref{den_para} hold.
Let
\begin{equation}
x=x_0+\frac{\xi}{n\rho(\theta)},\quad 
y=x_0+\frac{\eta}{n\rho(\theta)}.
\end{equation}
we known that $W_{\pm}$ are two saddle points of $\hat{F}(z;x_0)$, so
\begin{equation}
	\hat{F}_z(W_{\pm};x_0)=0,
\end{equation}
Straightforward calculation gives us
\begin{equation}
	\hat{F}_{zz}(W_{\pm};x_0)=\frac{(M-a)W_{\pm}-(M+1)}{W_{\pm}(W_{\pm}-1)(aW_{\pm}+1)}.
\end{equation}
Note that the contour $\Sigma_{vertical}$ depends on a parameter $\epsilon$. Applying the residue theorem, we evaluate $I_2$ as $\epsilon \to 0$, 
\begin{equation}
\begin{split}
	\lim_{\epsilon \to 0}I_2
	&=\frac{1}{y(2\pi i)^2}\int_{\mathcal{C}}ds\int_{\Sigma_{vertical}}dt  \frac{e^{F(s;y)}}{e^{F(t;x)}}\frac{1}{s-t}\\
	&=\frac{1}{y2\pi i}\int_{nW_+}^{nW_-}\frac{e^{F(s;y)}}{e^{F(s;x)}}ds\\
	&=\frac{1}{y2\pi i}\int_{nW_+}^{nW_-}\big(\frac{x}{y}\big)^s ds\\
	&=\frac{\big(y\log (x/y)\big)^{-1}}{2\pi i}\big((x/y)^{nW_+}-(x/y)^{nW_-}\big).
\end{split}
\end{equation}
 Here assume that $x\neq y$ and $n$ lagre enough, we have 
 so
\begin{equation}
	\frac{x}{y}=\frac{x_0+\frac{\xi}{n\rho(\theta)}}{x_0+\frac{\eta}{n\rho(\theta)}}=1+\frac{\xi-\eta}{n\rho(\theta)x_0}+O\Big(\frac{1}{\big(n\rho(\theta)x_0\big)^2}\Big),
\end{equation}
and 
\begin{equation}
	y\log(\frac{x}{y})=\big(x_0+\frac{\eta}{n\rho(\theta)}\big)\log\big(1+\frac{\xi-\eta}{n\rho(\theta)x_0}+O\big(\frac{1}{(n\rho(\theta)x_0)^2}\big)\Big)\big)
	=\frac{\xi-\eta}{n\rho(\theta)}\Big(1+O\big(\frac{1}{n\rho(\theta)x_0}\big)\Big).
\end{equation}
Combining the above approximations, if $\xi\neq \eta $, we obtain
\begin{equation}\label{I2sine}
\begin{split}	
	\lim_{\epsilon \to 0}I_2	&=\frac{\Big(\frac{\xi-\eta}{n\rho(\theta)}\big(1+O(\frac{1}{(n\rho(\theta)x_0)})\big)\Big)^{-1}}{2\pi i}\big(e^{\frac{(\xi-\eta) W_+}{\rho(\theta) x_0}}-e^{\frac{(\xi-\eta) W_-}{\rho(\theta) x_0}}\big)\big(1+O\big(\frac{1}{n\rho(\theta)x_0}\big)\big)\\ 
		&=\frac{n\rho(\theta)}{2\pi(\xi-\eta)}\big(e^{\frac{(\xi-\eta) W_+}{\rho(\theta) x_0}}-e^{\frac{(\xi-\eta) W_-}{\rho(\theta) x_0}}\big)+O\big(\frac{1}{n\rho(\theta)x_0}\big)\\ 
	    &=\frac{n\rho(\theta)\sin \big(\pi(\xi-\eta) \big)}{\pi(\xi-\eta)}e^{\pi \frac{\Re W_+}{\Im W_+}(\xi-\eta)+O\big(\frac{1}{n\rho(\theta)x_0}\big)}.
\end{split}
\end{equation}
 By analytic continuation argument we know that  \eqref{I2sine} also holds for $\xi = \eta$.

Later, we are going to estimate $I_1$ as $\epsilon \to 0$,
\begin{equation}\label{limitI_1}
\begin{split}
		\lim_{\epsilon \to 0}I_1
		&=\lim_{\epsilon \to 0}\frac{1}{y(2\pi i)^2}\int_{\mathcal{C}} \int_{\Sigma_{curved}} \frac{e^{F(s;y)}}{e^{F(t;x)}}\frac{1}{s-t}dsdt\\
		&=\frac{1}{y(2\pi i)^2}\lim_{\epsilon \to 0}\int_{n\tilde{\Sigma}^r \setminus (D_{\epsilon}(nW_{\pm}))}\big(\int_{\mathcal{C}}ds\frac{e^{F(s;y)}}{e^{F(t;x)}}\frac{1}{s-t}\big)dt\\
		&=\frac{1}{y(2\pi i)^2}\text{P.V.}\int_{n\tilde{\Sigma}^r }dt\int_{\mathcal{C}}ds\frac{e^{F(s;y)}}{e^{F(t;x)}}\frac{1}{s-t}
\end{split}
\end{equation}
where P.V. means the Cauchy principle value. 

Next, we will show that the main contribution to the Cauchy principle integral is from the local domain in the sense that remaining part of the integral is negligible in the asymptotic analysis.
To estimate \eqref{limitI_1} the limit of $I_1$, by setting $r_0:=\big(n\rho(\theta)\big)^{\frac{3}{5}}x_0$ we define 
\begin{equation}
	\mathcal{C}_{\text{local}}^{\pm}=\mathcal{C}\cap D_{r_0}(nW_{\pm}),\quad \Sigma_{local}^{\pm}=\Sigma \cap D_{r_0}(nW_{\pm}),
\end{equation}
where $D_r(a)$ the disc centered
at $a$ with radius $r$. 
For $s \in \mathcal{C}_{local}^{+}$ and $t \in \Sigma_{local}^{+}$, making the change of variables
\begin{equation}
	s=nW_{+}+\sqrt{n\rho(\theta)}x_0u, \quad	t=nW_{+}+\sqrt{n\rho(\theta)}x_0v,
\end{equation}
uniformly for $s \in D_{r_0}(nW_+)$, we have
\begin{align}\label{BULKeFs}
e^{F(s;y)}
&=e^{F(s;x_0)}\big(1+\frac{\eta}{n\rho(\theta)x_0}\big)^{-s} \notag\\
&=n^{(1-M/a)n}e^{\tilde{c}_M+n\hat{F}(W_{+}+\sqrt{\rho(\theta)/n}x_0u;x_0)}\big(1+\frac{\eta}{n\rho(\theta)x_0}\big)^{-s}\big(1+O\big(\big(n\rho(\theta)x_0\big)^{-\frac{1}{2}}\big)\big)\notag\\
&=n^{(1-M/a)n}e^{\tilde{c}_M+n\hat{F}(W_+;x_0)}e^{\frac{\rho(\theta)x_0}{2}\hat{F}_{zz}(W_+;x_0)u^2-\frac{W_+ \eta}{x_0\rho(\theta)}}\big(1+O\big(\big(n\rho(\theta)x_0\big)^{-\frac{1}{5}}\big)\big),
\end{align}
where \begin{equation}
	\tilde{c}_M= \sum_{j=1}^M(v_j+\frac{1}{2})\log\Big(\frac{W_+}{W_+ +1/a}\Big)+\frac{1}{2}\log\Big(\frac{W_+}{W_+ -1}\Big).
\end{equation}
A parallel argument yields that uniformly for $t\in D_{r_0}(nW_{+})$,
\begin{equation}\label{BULKeFt}
	e^{F(t;x)}=n^{(1-M/a)n}e^{\tilde{c}_M+n\hat{F}(W_+;x_0)}e^{\frac{\rho(\theta)x_0}{2}\hat{F}_{zz}(W_+;x_0)v^2-\frac{W_+ \xi}{x_0\rho(\theta)}}\big(1+O\big(\big(n\rho(\theta)x_0\big)^{-\frac{1}{5}}\big)\big).
\end{equation}
 Combining \eqref{BULKeFs} and \eqref{BULKeFt}, we have 
\begin{align}\label{es.p.v.local}
 &\text{P.V.} \int_{\mathcal{C}_{local}^{+}}ds \oint_{\Sigma_{local}^{+}}dt\frac{e^{F(s;y)}}{e^{F(t;x)}}\frac{1}{s-t} \notag\\
 &=\frac{1}{\sqrt{n\rho(\theta)}x_0}e^{-\frac{W_+(\xi-\eta)}{x_0\rho(\theta)}}\text{P.V.}\int_{\mathcal{C}_{local}^{+}}ds\oint_{\Sigma_{local}^{+}}dt\frac{e^{\frac{\rho(\theta)x_0}{2}\hat{F}_{zz}(W_+;x_0)(u^2-v^2)}}{u-v}\big(1+O\big(\big(n\rho(\theta)x_0\big)^{-\frac{1}{5}}\big)\big),
\end{align}
where on the right-hand side, we understand $u$ and $v$ as functions of $s$ and $t$, respectively. Note that the $O\big(\big(n\rho(\theta)x_0\big)^{-1/5}\big)$ term is uniform and analytic in $D_{r_0}(nW_+)$.
Next, the fllowing Lemmas about properties of $\hat{F}(z;x_0)$ on the contours will play a important role in our later analysis, whose proofs are left in the end of this
subsection.
\begin{lemma}\label{sublem1}
	For any $x_0 \in (0,x_*)$, which can be parameterized by $\theta_0 =\theta(x_0)$ and $W_+=W_+(x_0)$, there exist constans $\epsilon$ and $\delta$ such that
	\begin{equation}\label{sublem1ineq}
		\Re \hat{F}(z;x_0)\geq \Re \hat{F}(W_{\pm};x_0)+\epsilon|z-W_{\pm}| ,\quad z \in \tilde{\Sigma}^{\epsilon}\cap D_{\delta}(W_{\pm}).
	\end{equation} 
	Moreover, we have 
	\begin{equation}
		\frac{d\Re \hat{F}(\zeta(\theta);x_0)}{d \theta}
		\begin{cases}
			<0, &\text{for}\quad \theta\in (0,\theta_0),\\
			>0,  &\text{for}\quad \theta\in (\theta_0,\frac{\pi}{M+1}),\\
		\end{cases}
	\end{equation}
	\begin{equation}
		\frac{d\Re \hat{F}(\overline{\zeta(\theta)};x_0)}{d \theta}
		\begin{cases}
			<0, &\text{for}\quad \theta\in (0,\theta_0),\\
			>0,  &\text{for}\quad \theta\in (\theta_0,\frac{\pi}{M+1}).\\
		\end{cases}
	\end{equation}
\end{lemma}
\begin{lemma}\label{sublem2}
	For all conjugate pairs $w_{\pm}\in \mathbb{C}_{\pm}$ locating on $\tilde{\Sigma}$, there exists constants $\epsilon,\delta>0$ such that for all $ a\in \mathbb{R}$,
	\begin{equation}\label{sublem2ineq}
		\Re \hat{F}(\Re w_{\pm}+iy;a)\leq \Re \hat{F}(w_{\pm};a)-\epsilon|y-\Im w_{\pm}| ,\quad \text{for} |y-\Im w_{\pm}|<\delta .
	\end{equation} 
	Moreover,
	\begin{equation}
		\frac{d\Re \hat{F}(\Re w_+ +iy;a)}{d y}
		\begin{cases}
			<0, &\text{for}\quad y>\Im w_+ ,\\
			>0,  &\text{for}\quad y \in (0,\Im w_+).\\
		\end{cases}
	\end{equation}
	\begin{equation}
		\frac{d\Re \hat{F}(\Re w_- +iy;a)}{d y}
		\begin{cases}
			>0, &\text{for}\quad y<\Im w_- ,\\
			<0,  &\text{for}\quad y \in (\Im w_-,0).\\
		\end{cases}
	\end{equation}
\end{lemma}
By above lemmas, the value of $\hat{F}(s;x_0)$ attians its maximum over $\mathcal{C}$ around two  points $W_{\pm}$, while $\hat{F}(t;x_0)$ attains its minimum over  $\Sigma_{curved}$ around  $nW_{\pm}$. Then there is a constant $\epsilon_1$ such that for all $s \in \mathcal{C}_{local}^+$
\begin{equation}\label{slocal}
	|e^{\frac{\rho(\theta)x_0}{2}\hat{F}_{zz}(W_+;x_0)u^2}|\leq e^{-\epsilon_1|u|^2},
\end{equation}
and there is a constant $\epsilon_2$ such that for all $t \in \Sigma_{local}^+$,
\begin{equation}\label{tlocal}
	|e^{\frac{\rho(\theta)x_0}{2}\hat{F}_{zz}(W_+;x_0)v^2}|\geq e^{\epsilon_2|v|^2}.
\end{equation}
Combining \eqref{es.p.v.local}, \eqref{slocal} and \eqref{tlocal}, applying saddle point method, we obtain 
\begin{align}\label{pvlocal+}
	\text{P.V.} \int_{\mathcal{C}_{local}^{+}}ds \oint_{\Sigma_{local}^{+}}dt\frac{e^{F(s;y)}}{e^{F(t;x)}}\frac{1}{s-t}
	=O\Big(\big(n\rho(\theta)\big)^{\frac{3}{5}}x_0\Big).
\end{align}
In the similar manner, after making change of the variables 
\begin{equation}
s=nW_{-}+\sqrt{n\rho(\theta)}x_0u, \quad	t=nW_{-}+\sqrt{n\rho(\theta)}x_0v,
\end{equation}
we have 
\begin{equation}\label{pvlocal-}
\text{P.V.} \int_{\mathcal{C}_{local}^{-}}ds \oint_{\Sigma_{local}^{-}}dt\frac{e^{F(s;y)}}{e^{F(t;x)}}\frac{1}{s-t}=O\Big(\big(n\rho(\theta)\big)^{\frac{3}{5}}x_0\Big).
\end{equation}
\begin{lemma}\label{stglobales}
	For $s \in \mathcal{C}\setminus (\mathcal{C}_{local}^{+}\cup\mathcal{C}_{local}^{-})$ and $t \in \Sigma \setminus (\Sigma_{local}^{+}\cup \Sigma_{local}^{-})$, thers exists $\epsilon_3$  such that for large enough $n$
\begin{equation}
	|e^{F(s;y)}|=|e^{F(s;x_0)}\big(1+\frac{\eta}{n\rho(\theta)x_0}\big)^{-s}|<|e^{F(nW_{\pm};x_0)}|e^{-\epsilon_3\big(n\rho(\theta)\big)^{\frac{1}{5}}},
\end{equation}
\begin{equation}
	|e^{-F(t;x)}|=|e^{-F(t;x_0)}\big(1+\frac{\xi}{n\rho(\theta)x_0}\big)^{t}|<|e^{-F(nW_{\pm};x_0)}|e^{-\epsilon_3\big(n\rho(\theta)\big)^{\frac{1}{5}}}.
\end{equation}
\end{lemma}
Lemma \ref{stglobales} is derived from Lemmas \ref{sublem1} and \ref{sublem2}, and we will not delve into the specific details of the proof. It is advisable to consult Lemma 2.1 in \cite{LWZ16}.
Applying Lemma \ref{stglobales} to $I_2$ , we have 
\begin{equation}\label{pvglo}
	\text{P.V.}\int_{ \mathcal{C}\setminus (\mathcal{C}_{local}^{+}\cup\mathcal{C}_{local}^{-})}ds\oint_{\Sigma \setminus (\Sigma_{local}^{+}\cup \Sigma_{local}^{-})}dt\frac{e^{F(s;y)}}{e^{F(t;x)}}\frac{1}{s-t}=O\Big(e^{-\epsilon \big(n\rho(\theta)\big)^{\frac{1}{5}}}\Big).
\end{equation}
Combining the estimates \eqref{pvlocal+}, \eqref{pvlocal-} and \eqref{pvglo}, we obtain
\begin{equation}\label{I_2GLO}
\begin{split}	|\lim_{\epsilon \to 0}I_1|&=\frac{e^{\pi \frac{\Re W_+}{\Im W_+}(\xi-\eta)}}{|y|}\Big(O\Big(\big(n\rho(\theta)\big)^{\frac{3}{5}}x_0\Big)+O\Big(e^{-\epsilon \big(n\rho(\theta)\big)^{\frac{1}{5}}}\Big)\Big)\\
	&=O\Big(\big(n\rho(\theta)\big)^{\frac{3}{5}}e^{\pi \frac{\Re W_+}{\Im W_+}(\xi-\eta)}\Big).\end{split}
\end{equation}
With the aid of  \eqref{I2sine} and \eqref{I_2GLO}, we have $\frac{|\lim\limits_{\epsilon \to 0}I_1|}{|\lim\limits_{\epsilon \to 0}I_2|}=O\Big(\big(n\rho(\theta)\big)^{-\frac{2}{5}}\Big) $.

Summing up $I_1$ and $I_2$ as $\epsilon \to 0$, and letting $n \to \infty$ and $\Delta_{M,n} \to 0$, we complete the proof of Theorem \ref{TBU}  $\mathtt{(ii)}$.
\end{proof}
\begin{proof}[Proof of Lemma \ref{sublem1}]
	Recall that $x_0 = x(\theta_0)$ is given by \eqref{x_para}. For all $\theta \in (0, \pi/(M+1))$, the function $\zeta(\theta)$ satisfies the equation
	\begin{equation}
		\zeta(\theta)^{M+1} - (\zeta(\theta)-1)(\zeta(\theta)+1/a)^M = 0. \label{eq:zeta_eq}
	\end{equation}
	Consequently, the derivative of $\hat{F}(\zeta(\theta); x_0)$ with respect to $\theta$ is
	\begin{align}
		\frac{d}{d\theta}\hat{F}(\zeta(\theta);x_0) = \frac{d\hat{F}}{d\zeta}\frac{d\zeta}{d\theta}
		&= \log\left(\frac{\zeta(\theta)^{M+1}}{(\zeta(\theta)-1)(\zeta(\theta)+1/a)^M}\right) \frac{d\zeta}{d\theta} \notag \\
		&= \log\left(\frac{x(\theta)}{x(\theta_0)}\right) \frac{d\zeta}{d\theta}. \label{eq:F_deriv}
	\end{align}
	
	The function $\theta \mapsto \sin \theta / \sin(c\theta)$ is strictly decreasing on $(0, \pi)$ for $0 < c < 1$, which implies $dx(\theta)/d\theta < 0$. Since $x(\theta) > 0$ and
	\begin{equation}
		\Re\left( \frac{d\zeta(\theta)}{d\theta} \right) = \frac{d}{d\theta}\left( \frac{\sin ((M+1)\theta)}{(a+1)\sin (M\theta) - a\sin((M+1)\theta} \right) < 0 \label{eq:dz_real}
	\end{equation}
	holds for all $\theta \in (0, \pi/(M+1))$, we deduce that
	\begin{equation}
		\frac{d}{d\theta} \Re \hat{F}(\zeta(\theta);x_0)
		\begin{cases}
			< 0, & \text{for } \theta \in (0, \theta_0), \\
			> 0, & \text{for } \theta \in (\theta_0, \pi/(M+1)).
		\end{cases} \label{eq:dReF_sign}
	\end{equation}
	Moreover,
	\begin{equation}
		\left. \frac{d^2}{d\theta^2} \Re \hat{F}(\zeta(\theta);x_0) \right|_{\theta=\theta_0} > 0. \label{eq:d2ReF_sign}
	\end{equation}
	Properties \eqref{eq:dReF_sign} and \eqref{eq:d2ReF_sign} establish the inequality \eqref{sublem1ineq} for $z \in \mathbb{C}_{+}$.
	
	Finally, since $\Re \hat{F}(z;x)$ is symmetric with respect to the real axis, the result holds for $z \in \mathbb{C}_{-}$ as well, completing the proof.
\end{proof}
\begin{proof}[Proof of Lemma \ref{sublem2}]
	
For $\theta \in (0, \pi/(M+1))$, we observe that
	\begin{equation}
		\Re W_{\pm}  \in \left(0, \frac{M+1}{M-a}\right).
	\end{equation}
	For $y > 0$, the second derivative satisfies
	\begin{equation}
		\frac{d^2}{dy^2}\Re \hat{F}(x+iy;a) 
		= -\left.\Re\frac{d^2\hat{F}(z;a)}{dz^2}\right|_{z=x+iy},
	\end{equation}
	where
	\begin{equation}
		\frac{d^2\hat{F}(z;a)}{dz^2} = \frac{M+1}{z} - \frac{1}{z-1} - \frac{M}{z+1/a}.
	\end{equation}
	Since 
	\begin{equation}
		\lim_{y \to +\infty} \frac{d^2}{dy^2}\Re \hat{F}(x+iy;a) = 0,
	\end{equation}
	it follows that  the equation $\frac{d^2}{dy^2}\Re \hat{F}(x+iy;a) = 0$ has a unique real root $y^* \in [0, \infty)$ for each $x \in (0, (M+1)/(M-a))$. Consequently, $\frac{d}{dy}\Re \hat{F}(\Re W_{\pm} + iy; a)$ is strictly increasing on $(0, y^*)$ and strictly decreasing on $(y^*, \infty)$.  
	
	Note that $\frac{d}{dy}\Re \hat{F}(\Re W_{\pm } + iy; a)$ is a continuous odd function in $y$ that vanishes at $y = 0$. So we known that  $\frac{d}{dy}\Re \hat{F}(\Re W_{\pm} + iy; a) > 0$ for $y \in (0, y^*)$. Furthermore, 
	\begin{equation}
		\left.\frac{d}{dy}\Re \hat{F}(\Re W_{+} + iy; a)\right|_{y=\Im W_{+}} 
		= -\Im  \left.\frac{d\hat{F}}{dz}(z;a)\right|_{z=W_{+}} 
		= 0,
	\end{equation}
	which implies $\Im W_{+} \in (y^*, \infty)$. On $[0, \infty)$, we conclude that the function $\Re \hat{F}(\Re W_{\pm} + iy; a)$ is concave and attains its maximum at $y = \Im W_+$, which give us inequality \eqref{sublem2ineq}. 
	
	By symmetry in $y$, the analogous result holds for $y < 0$. This completes the proof.
\end{proof}

\subsection{Proof of Theorem \ref{TSEU}}
\begin{proof}[Sketched proof of Theorem \ref{TSEU}]
We provide a sketch of the proof for fixed $M$, highlighting refinements to the argument in Theorem \ref{subthm}. Adopting the scaling 
\begin{equation} \label{eq:scaling}
	x = x_* + \frac{c_2\xi}{n^{2/3}}, \quad 
	y = x_* + \frac{c_2\eta}{n^{2/3}}, \quad \xi,\eta \in \mathbb{R},
\end{equation}
where $x_*$ is defined in \eqref{defxright}, the correlation kernel transforms to
\begin{equation} \label{eq:kernel_transformed}
	\widetilde{K}_n(x,y) = \frac{1}{y(2\pi i)^2} \int_{\mathcal{C}} ds \oint_{\Sigma} dt  \frac{e^{F(s;x_*)}}{e^{F(t;x_*)}} \frac{(1 + n^{-2/3}c_1^{-1}\eta)^{-s}}{(1 + n^{-2/3}c_1^{-1}\xi)^{-t}} \frac{1}{s-t}.
\end{equation}
Here $F(s;x)$ is given by \eqref{FBS}, and contours $\mathcal{C}$, $\Sigma$ are deformed such that $\Sigma$ lies to the left of $\mathcal{C}$. 

Define the saddle point $z_0 = \frac{M+1}{M-a}$. Through the change of variables
\begin{equation} \label{eq:contour_scaling}
	s = n z_0 + c_1 u n^{2/3}, \quad 
	t = n z_0 + c_1 v n^{2/3}, \quad 
	c_1 := \frac{x_*}{c_2},
\end{equation}
taking a Taylor expansion at $z_0$, we have 
\begin{equation} \label{eq:taylor_expansion}
	\begin{split}
		\hat{F}(z_0 + n^{-1/3}c_1 u; x_*) 
		&= \hat{F}(z_0;x_*) + \hat{F}_z(z_0;x_*) c_1 u n^{-1/3} \\
		&\quad + \frac{1}{2} \hat{F}_{zz}(z_0;x_*) c_1^2 u^2 n^{-2/3} \\
		&\quad + \frac{1}{6} \hat{F}_{zzz}(z_0;x_*) c_1^3 u^3 n^{-1} + O(n^{-6/5}).
	\end{split}
\end{equation}
Direct computations yield the derivatives
\begin{equation} \label{eq:derivatives}
	\hat{F}_z(z_0;x_*) = 0, \quad 
	\hat{F}_{zz}(z_0;x_*) = 0, \quad 
	\hat{F}_{zzz}(z_0;x_*) = \frac{(M-a)^4}{M(M+1)(a+1)^2}.
\end{equation}
Substituting \eqref{eq:derivatives} into \eqref{eq:taylor_expansion} gives the uniform expansion for $u \in D_{n^{1/30}}(0)$,
\begin{equation} \label{eq:critical_expansion}
	\hat{F}(z_0 + n^{-1/3}c_1 u; x_*) = \hat{F}(z_0;x_*) + \frac{u^3}{3n} + O(n^{-6/5}).
\end{equation}

Deforming $\mathcal{C}$ and $\Sigma$ to pass through neighborhoods of $nz_0$ in  proper directions,  the contribution from the ``local" part is


\begin{equation}\label{airyes}
	\begin{split}
	&\frac{1}{y(2\pi i)^2}\int_{\mathcal{C}_{local}}ds\oint_{\Sigma_{local}}dt \frac{e^{F(s;x_*)}\Big(1+n^{-2/3}c_1\eta\Big)^{-s}}{e^{F(t;x_*)}\Big(1+n^{-2/3}c_1\xi\Big)^{-t}}\frac{1}{s-t}\\
	&=\frac{e^{n^{1/3}\frac{M+1}{M-a}c_1^{-1}(\xi-\eta)}}{n^{-2/3}c_1^{-1}x_*}\Big(\frac{1}{(2\pi i)^2}\int_{\mathcal{C}_<} du \int_{\Sigma_>} dv \frac{e^{\frac{u^3}{3}-u\eta}}{e^{\frac{v^3}{3}-v\xi}}\frac{1}{u-v}+O(n^{-1/5})\Big)\\
	&=\frac{e^{n^{1/3}\frac{M+1}{M-a}c_1^{-1}(\xi-\eta)}}{n^{-2/3}c_2}\Big(	K_{\mathtt{Ai}}(\xi,\eta)++O(n^{-1/5})\Big),
	\end{split}
\end{equation}
where $\mathcal{C}_<$ and $\Sigma_>$ are the images of $\mathcal{C}_{\text{local}}$ and $\Sigma_{\text{local}}$, similar to what is shown in the proof of Theorem \ref{subthm}. 

Combining  the remaining part of the integral is negligible in the asymptotic analysis, we get the desired result.\end{proof}

\begin{acknow}
This work was supported by the National Natural Science Foundation of China \#12371157  and \#12090012.
\end{acknow}

	\end{document}